\documentclass[11pt]{amsart}
\usepackage{amsmath}
\usepackage{amssymb}
\usepackage{amscd}
\usepackage{amsthm}

\usepackage[all]{xy}

\newtheorem{thm}{Theorem}[section]
\newtheorem{lem}[thm]{Lemma}
\newtheorem{cor}[thm]{Corollary}
\newtheorem{conj}[thm]{Conjecture}
\newtheorem{prop}[thm]{Proposition}

\theoremstyle{remark}

\newtheorem{rem}[thm]{Remark}

\theoremstyle{definition}
\newtheorem{defn}[thm]{Definition}

\numberwithin{equation}{section}

\allowdisplaybreaks[4]

\DeclareMathOperator{\Ext}{Ext}
\DeclareMathOperator{\EXT}{{\mathcal E}xt}
\DeclareMathOperator{\Pic}{Pic}

\DeclareMathOperator{\supp}{supp}

\DeclareMathOperator{\Bs}{Bs}

\DeclareMathOperator{\IM}{Im}

\DeclareMathOperator{\End}{End}

\DeclareMathOperator{\Exec}{Exec}

\begin{document}

\vfuzz0.5pc
\hfuzz0.5pc 

\newcommand{\claimref}[1]{Claim \ref{#1}}
\newcommand{\thmref}[1]{Theorem \ref{#1}}
\newcommand{\propref}[1]{Proposition \ref{#1}}
\newcommand{\lemref}[1]{Lemma \ref{#1}}
\newcommand{\coref}[1]{Corollary \ref{#1}}
\newcommand{\remref}[1]{Remark \ref{#1}}
\newcommand{\conjref}[1]{Conjecture \ref{#1}}
\newcommand{\questionref}[1]{Question \ref{#1}}
\newcommand{\defnref}[1]{Definition \ref{#1}}
\newcommand{\secref}[1]{Sec. \ref{#1}}
\newcommand{\ssecref}[1]{\ref{#1}}
\newcommand{\sssecref}[1]{\ref{#1}}

\newcommand{\RED}{{\mathrm{red}}}
\newcommand{\tors}{{\mathrm{tors}}}
\newcommand{\eq}{\Leftrightarrow}

\newcommand{\mapright}[1]{\smash{\mathop{\longrightarrow}\limits^{#1}}}
\newcommand{\mapleft}[1]{\smash{\mathop{\longleftarrow}\limits^{#1}}}
\newcommand{\mapdown}[1]{\Big\downarrow\rlap{$\vcenter{\hbox{$\scriptstyle#1$}}$}}
\newcommand{\smapdown}[1]{\downarrow\rlap{$\vcenter{\hbox{$\scriptstyle#1$}}$}}

\newcommand{\A}{{\mathbb A}}
\newcommand{\I}{{\mathcal I}}
\newcommand{\J}{{\mathcal J}}
\newcommand{\CO}{{\mathcal O}}
\newcommand{\C}{{\mathcal C}}
\newcommand{\BC}{{\mathbb C}}
\newcommand{\BQ}{{\mathbb Q}}
\newcommand{\m}{{\mathcal M}}
\newcommand{\h}{{\mathcal H}}
\newcommand{\Z}{{\mathcal Z}}
\newcommand{\BZ}{{\mathbb Z}}
\newcommand{\W}{{\mathcal W}}
\newcommand{\Y}{{\mathcal Y}}
\newcommand{\T}{{\mathcal T}}
\newcommand{\BP}{{\mathbb P}}
\newcommand{\CP}{{\mathcal P}}
\newcommand{\G}{{\mathbb G}}
\newcommand{\BR}{{\mathbb R}}
\newcommand{\D}{{\mathcal D}}
\newcommand {\LL}{{\mathcal L}}
\newcommand{\f}{{\mathcal F}}
\newcommand{\E}{{\mathcal E}}
\newcommand{\BN}{{\mathbb N}}
\newcommand{\N}{{\mathcal N}}
\newcommand{\K}{{\mathcal K}}
\newcommand{\R} {{\mathbb R}}
\newcommand{\PP} {{\mathbb P}}
\newcommand{\BF}{{\mathbb F}}
\newcommand{\closure}[1]{\overline{#1}}
\newcommand{\EQ}{\Leftrightarrow}
\newcommand{\imply}{\Rightarrow}
\newcommand{\isom}{\cong}
\newcommand{\embed}{\hookrightarrow}
\newcommand{\tensor}{\mathop{\otimes}}
\newcommand{\wt}[1]{{\widetilde{#1}}}
\newcommand{\ol}{\overline}
\newcommand{\ul}{\underline}

\newcommand{\bs}{{\backslash}}
\newcommand{\CS}{{\mathcal S}}
\newcommand{\CA}{{\mathcal A}}
\newcommand{\Q} {{\mathbb Q}}
\newcommand{\F} {{\mathcal F}}
\newcommand{\sing}{{\text{sing}}}
\newcommand{\U} {{\mathcal U}}
\newcommand{\B}{{\mathcal B}}

\title{Self Rational Maps of $K3$ Surfaces}
\author{Xi Chen}
\address{632 Central Academic Building\\
University of Alberta\\
Edmonton, Alberta T6G 2G1, CANADA}
\email{xichen@math.ualberta.ca}
\date{\today}

\thanks{Research partially supported by a grant from the Natural Sciences
and Engineering Research Council of Canada.}



\begin{abstract}
We prove that a very general projective $K3$ surface does not admit a dominant self rational map of degree at least two.
\end{abstract}

\maketitle

\section{Introduction}\label{SEC001}

The purpose of this paper is to prove the following conjecture:

\begin{thm}\label{THM001}
There is no dominant self rational map $\phi: X\dashrightarrow X$ of degree $\deg\phi > 1$
for a very general projective $K3$ surface $X$ of genus $g\ge 2$.
\end{thm}

For the background of this conjecture, please also see \cite{D}.

Self rational maps of $K3$ surfaces arise naturally in several contexts. There are special $K3$ surfaces
with nontrivial self rational maps. Here are two typical examples \cite{D}:
\begin{itemize}
\item if $X$ is an elliptic $K3$ surface, i.e., a $K3$ surface admitting an elliptic fiberation $X/\PP^1$, there are
self rational maps $\phi: X\dashrightarrow X$ of degree $\deg\phi > 1$ mapping $X/\PP^1$ to $X/\PP^1$ fiberwisely;
\item if $X$ is a Kummer surface, i.e., a $K3$ surface birational to the quotient of an abelian surface by an
involution, there are self rational maps $\phi: X\dashrightarrow X$ of $\deg\phi > 1$ descended from the abelian surface.
\end{itemize}
To our knowledge, these are the only special $K3$ surfaces known to have nontrivial self rational maps. It would be
interesting to find others.

More generally, every variety $X$ birational to a projective family of abelian varieties over some base $B$ admits nontrivial self-rational maps by fixing a multi-section $L\subset X$ of $X/B$ with degree $n$ and sending a point $x\in X_b$ to
$L - (n-1) x$. This also works if $X/B$ is birational to a fiberation of quotients of abelian varieties by finite groups.

For a $K3$ surface $X$ over a number field $k$,
the existence of self rational maps for $X$ is closely related to the arithmetic problem on the potential density of
$k$-rational points on $X$. If there is a rational map $\phi: X\dashrightarrow X$ of $\deg\phi > 1$ over a
finite extension $k'\to k$ of the base field, by
iterating $\phi$, we can produce many $k'$-rational points on $X$. Under suitable conditions, these $k'$-rational
points are Zariski dense in $X$ \cite{A-C}.

The existence of self rational maps of a $K3$ surface is also related to its hyperbolic geometry. Algebraic
surfaces that are holomorphically dominable by $\BC^2$ were classified by G. Buzzard and S. Lu \cite{B-L}.
They almost gave a complete answer except for the case of $K3$ surfaces.
They showed that elliptic $K3$ and Kummer surfaces are dominable by $\BC^2$. However,
it is unknown whether a generic $K3$ surface $X$ is dominable by $\BC^2$. It is no coincidence that elliptic $K3$ and
Kummer surfaces, as the examples of $K3$ surfaces admitting nontrivial self rational maps,
are dominated by $\BC^2$.
Indeed, if there exists a rational map $\phi: X\dashrightarrow X$ with some
dilating properties, then by iterating $\phi$ and taking the limit, we can arrive at a dominating meromorphic map
$\BC^2\dashrightarrow X$ \cite{C}.

So it becomes a natural question to ask whether a generic $K3$ surface admits a nontrivial self rational map.
Here by ``generic'', we mean ``very general'', i.e., a $K3$ surface represented by a point in the moduli space of
polarized $K3$ surfaces with countably many proper subvarieties removed. Needless to say, the hypothesis of $X$ being
very general is necessary since elliptic $K3$'s are parametrized by countably many hypersurfaces in the
moduli space. It also means that we have to use this hypothesis in an essential way.

A natural way to prove \thmref{THM001} is via degeneration.
Fortunately for us, there are good degenerations of $K3$ surfaces. Every $K3$ surface can be degenerated to a union
of rational surfaces. For example, a quartic $K3$ in $\PP^3$ can be degenerated to a union of two quadrics or four
planes and so on. To see how this can be done in general, we start with a union $W_0 = S_1\cup S_2$
of two Del Pezzo surfaces meeting transversely along a smooth elliptic curve $D$ (see \ssecref{SS001} for details).
Using the argument in \cite{CLM}, we can show that the natural map
\begin{equation}\label{E151}
\begin{split}
\Ext(\Omega_{W_0}, \CO_{W_0}) \xrightarrow{} H^0(T^1(X_0)) &= H^0(\EXT(\Omega_{W_0}, \CO_{W_0}))\\
&= H^0(\CO_D(-K_{S_1} - K_{S_2}))
\end{split}
\end{equation}
is surjective. Consequently, a general deformation of $W_0$ smooths out its singularities along $D$. And since the
dualizing sheaf $\omega_{W_0}$ of $W_0$ is trivial and $W_0$ is simply connected, $W_0$ can be deformed to a complex
$K3$ surface (not necessarily projective). If we further assume that $W_0$ possesses an indivisible ample line bundle
$L$ with $L^2 = 2g - 2$, then we can deform $W_0$ while ``preserving'' $L$ and thus deform $W_0$ to a smooth projective $K3$
surface of genus $g$. Since the moduli space of polarized $K3$ surfaces with a fixed genus $g$ is irreducible, this
argument shows that every polarized $K3$ surface $(S, L)$ can be degenerated to $(S_1\cup S_2, L)$
described as above.

Note that $W_0$ is constructed by gluing $S_1$ and $S_2$ transversely along $D$ via two immersions $i_k:
D\hookrightarrow S_k$ for $k=1,2$. A line bundle $L$ on $W_0$ is given by two line bundles $L_k\in \Pic(S_k)$ such
that $L_1$ and $L_2$ agrees on $D$, i.e.,
\begin{equation}\label{E180}
i_1^* L_1 = i_2^* L_2.
\end{equation}

Naturally, we expect that the existence of self rational maps for generic $K3$ surfaces will induce a self
rational map for such a union $S_1 \cup S_2$.
To be more precise, we let $W/\Delta$ be a family of $K3$ surfaces of genus $g$ over the
disk $\Delta = \{ |t| < 1\}$ whose central fiber $W_0$ is a union $S_1\cup S_2$ given as above.
It turns out that $W$ is singular and we need to work with a resolution $X$ of singularities of $W$
(see \ssecref{SS004}).
Suppose that there are rational maps $\phi_t: X_t\dashrightarrow X_t$ for all $t\ne 0$. It is easy to see that
$\phi_t$ can be extended to a rational map $\phi: X\dashrightarrow X$ after a base change. 
Basically, we are trying to study the self-rational maps $\phi_t$ by studying $\phi_0$.
However, the rational map $\phi_0: S_1\cup S_2 \dashrightarrow S_1\cup S_2$ does not tell us much itself
because, among other things,
\begin{itemize}
\item $S_k$ might very well be contracted by $\phi$, although it does not turn out to be the case (see
\propref{PROP020});
\item $\phi$ might not be regular along $D = S_1\cap S_2$, i.e., $D$ is contained in the indeterminate locus of $\phi$.
\end{itemize}
To really understand the rational map $\phi$, we need to resolve the indeterminacy of $\phi$ first. Namely, 
there exists a birational regular map $f: Y\to X$ such that $\varphi = \phi\circ f$ is regular with the
commutative diagram
\begin{equation}\label{E906}
\xymatrix{
Y \ar[r]^\varphi\ar[d]_f & X\\
X\ar@{-->}[ur]_\phi
}
\end{equation}
We can make $Y_0$ as ``nice'' as possible by the stable reduction theorem in \cite{KKMS},
although that comes at a
cost that we may get many ``irrelevant'' components of $Y_0$ that are contracted by $\varphi$.
By adjunction and Riemann-Hurwitz, we can
figure out all the ``relevant'' components of $Y_0$, which turn out to be
the union $\CS$ of the components of $E\subset Y_0$ with discrepancy $a(E, X) = 0$ under $f$
(see \ssecref{SS003} for details). This occupies the first part of our proof.

We construct $Y$ as a resolution of indeterminacy of $\phi$. Alternatively and equivalently, we
can also construct $Y$ as follows.
Fixing a sufficiently ample divisor $\LL$ on $X$, we can construct $Y$ by resolving the
base locus $\Bs(f_* \varphi^* \LL)$ of
the linear series $f_* \varphi^*\LL$ as $\LL$ varies in $|\LL|$.
It is not hard to see that $\Bs(f_* \varphi^* \LL)$, which is merely the indeterminacy of $\phi$,
is independent of our choice of $Y$ and $\LL$.
 
It turns out that   
\begin{equation}\label{E181}
\supp f_*(\varphi^* \LL \cap \T)\subset \Bs(f_* \varphi^*\LL_0)
\end{equation}
where $\T = Y_0 - \CS$.
A large part of this paper is devoted to
the study of the curve $f_*(\varphi^*\LL\cap \T)$. As $\Bs(f_* \varphi^* \LL)$, this
curve does not depend on our choice of $Y$ and $\LL$.
It is ``rigid'' in the sense that it has only countably
many possible configurations.
More precisely, it is contained in a union $\Sigma$ of countably many rational curves on $X_0$.
This union $\Sigma$ is determined completely by the Kodaira-Spencer class
of $W$. Indeed, it is determined by the $T^1$ class of $W$, i.e., the
singularities of $W$ lying on $D$. In particular, it does not depend on $\phi$. So by iterating $\phi$,
we can show that some components of $\Sigma$ are contracted or mapped onto some other components of $\Sigma$,
which leads to a proof of the main theorem. 

One of the crucial facts employed in our proof is
\begin{equation}\label{E051}
\End(D) = \BZ
\end{equation}
where $\End(D)$ is the ring of the endomorphisms of $D$ with a fixed point.
This holds because $D$ is an elliptic curve
of general moduli and hence carries no complex multiplication.
In some sense, the triviality of self rational maps of a
general $K3$ surface comes down to the triviality of self
rational maps of a general elliptic curve. Since elliptic curves are customarily regarded as
Calabi-Yau (CY) manifolds of dimension one, this suggests that the same holds in higher dimension
and there might be a way to prove it inductively, at least for CY manifolds which
are complete intersections in $\PP^n$. We will propose
the following conjecture for quintic threefolds but say no more.

\begin{conj}\label{CONJ001}
A very general quintic threefold $X$ in $\PP^4$ does not admit a self rational map
$\phi: X\dashrightarrow X$ of degree $\deg \phi > 1$.
\end{conj}

\medskip\noindent{\bf Conventions.}
We work exclusively over $\BC$ and with analytic topology wherever possible.
Clearly, \thmref{THM001} fails trivially in positive characteristic.
A $K3$ surface in this paper, unless specified otherwise, is always projective.

\medskip\noindent{\bf Acknowledgments.}
I am grateful to Jason Starr for introducing me to the problem.

\section{Degeneration of $K3$ Surfaces and Resolution of Indeterminacy}\label{SEC002}

\subsection{Degeneration of $K3$ surfaces}\label{SS001}

Let $W/\Delta$ be a family of $K3$ surfaces over the disk $\Delta$
whose general fibers $W_t$ are general $K3$ surfaces of genus $g\ge 2$ and whose central fiber is a union of two
smooth rational surfaces $W_0 = S_1\cup S_2$ meeting transversely along a smooth elliptic curve $D = S_1\cap S_2$.
And there is an indivisible line bundle $L$ on $W/\Delta$ with $L_t^2 = 2g -2$ that polarizes $W_t$ for $t\ne 0$.

We choose $S_i$ with the following properties for $i=1,2$:
\begin{itemize}
\item the anti-canonical divisors $-K_{S_i}$ are ample, i.e., $S_i$ are Del Pezzo surfaces and $D\in |-K_{S_i}|$;
\item $K_{S_1}^2 = K_{S_2}^2$;
\item there are ample line bundles $L_i$ on $S_i$ such that
\begin{equation}\label{E150}
L_1\bigg|_D = L_2\bigg|_D, L\bigg|_{S_i} = L_i, L_i^2 = g - 1\text{ and } L_i D = g + 1.
\end{equation}
\end{itemize}
As outlined in the previous section, we can show that such a union $S_1\cup S_2$ can be deformed to a $K3$
surface of genus $g$.

Such surfaces $S_1$ and $S_2$ can be chosen in many different ways.
We use the degeneration in \cite{CLM} and \cite{Ch}: 
\begin{itemize}
\item if $g$ is odd, we let $S_i \isom \BF_0 = \PP^1\times \PP^1$ and
\begin{equation}\label{E900}
L_i = L\bigg|_{S_i} = C_i + \frac{g-1}{2} F_i
\end{equation}
where $C_i$ and $F_i$ are the generators of $\Pic(S_i)$ with $C_i^2 = F_i^2 = 0$ and
$C_iF_i = 1$ for $i=1,2$;
\item
if $g\ge 4$ is even, we let $S_i\isom\BF_1 = \PP(\CO_{\PP^1} \oplus \CO_{\PP^1}(-1))$ and
\begin{equation}\label{E440}
L_i = L\bigg|_{S_i} = C_i + \frac{g}{2} F_i
\end{equation}
where $C_i$ and $F_i$ are the generators of $\Pic(S_i)$ with $C_i^2 = -1$, $F_i^2 = 0$ and $C_iF_i = 1$ for $i=1,2$;
\item if $g = 2$, we take $S_i\isom \PP^2$ and $L_i = \CO_{S_i}(1)$.
\end{itemize}


For a general choice of $W$, the deformation theory tells us \cite{CLM} that $W$ has $\lambda = K_{S_1}^2 +
K_{S_2}^2$ rational double points $p_1, p_2, ..., p_\lambda\in D$ satisfying
\begin{equation}\label{E901}
\CO_D(\sum_{j=1}^{\lambda} p_j) = N_{D/S_1} \tensor N_{D/S_2} = \CO_D(-K_{S_1} - K_{S_2})
\end{equation}
where $N_{D/S_i}$ are the normal bundles of $D$ in $S_i$ for $i=1,2$ and
\begin{equation}\label{E189}
\lambda =
\begin{cases}
16 & \text{if } g\ge 3\\
18 & \text{if } g = 2.
\end{cases}
\end{equation}
That is, $W$ is locally given by
\begin{equation}\label{E903}
xy = t z
\end{equation}
at a point $p\in \Lambda = \{p_1,p_2,...,p_\lambda\}$,
where the surfaces $S_i$ are locally given by $S_1 = \{ x = t = 0\}$ and $S_2 = \{ y = t = 0\}$, respectively.

\begin{rem}\label{REM001}
For a general choice of $W$, we have a left exact sequence
\begin{equation}\label{E441}
0\xrightarrow{} \BZ^{\oplus 2} \xrightarrow{} \Pic(S_1)\oplus \Pic(S_2)
\oplus \BZ^{\oplus \lambda} \xrightarrow{\rho} \Pic(D)
\end{equation}
where $\rho$ is the map given by
\begin{equation}\label{E468} 
\rho(M_1, M_2, n_1, n_2,..., n_\lambda) = M_1 - M_2 + \sum n_i p_i
\end{equation}
with kernel freely generated by $(L_1, L_2, 0,0,...,0)$ and $(K_{S_1}, -K_{S_2}, 1,1,...,1)$.
\end{rem}

We can resolve the singularities of $W$ by blowing up $W$ along $S_1$. Let $X\to W$ be the blowup.
It is not hard to see that the central fiber of $X/\Delta$ is $X_0 = R_1\cup R_2$, where
$R_1$ is the blowup of $S_1$ at $\Lambda$ and $R_2\isom S_2$. At 
each point $p\in \Lambda$, $X$ is a small resolution of $p\in W$.
Note that a small resolution of a 3-fold rational double point usually results in a non-K\"ahler complex manifold.
However, in this case, $X$ is obviously projective over $\Delta$ since $W$ is projective over $\Delta$ and
$X$ is obtained from $W$ by blowing up along a closed subscheme. More explicitly, we have the line bundle
$lL - R_1$ on $X$ which is relatively ample over $\Delta$ for some large $l$.
Here we continue to use $D$ and $L$ to denote the intersection $R_1\cap R_2$ and the pullback of $L$ from $W$ to $X$,
respectively.

\subsection{Generality of $W$}\label{SS007}

Of course, we choose $W$ to be very general. Actually, we can be very precise on how general
$W$ should be. We pick $W$ such that
\begin{itemize}
\item $D$ satisfies \eqref{E051} and
\item \eqref{E441} holds, or equivalently, we have a left exact sequence
\begin{equation}\label{E628}
\xymatrix{
0 \ar[r] &  \Pic(X_0) \ar[r] & \Pic(R_1) \oplus \Pic(R_2) \ar[r]^(0.65){\rho} & \Pic(D)\\
& \BZ^{\oplus 2} \ar@{=}[u] & \BZ^{\oplus 20} \ar@{=}[u]}
\end{equation}
where $\rho(M_1, M_2) = M_1 - M_2$ for $M_i\in\Pic(R_i)$ and the kernel of $\rho$ is
freely generated by $(L_1, L_2)$ and $(K_{R_1}, - K_{R_2})$.
\end{itemize}
These requirements on $W$ (actually on $X_0 = R_1\cup R_2$)
are all we need to make our later argument work.

\subsection{Small resolutions of rational double points and flops}\label{SS004}

Of course, we may also resolve the singularities of $W$ by blowing up $W$ along $S_2$
with $\widehat X$ the resulting 3-fold. Indeed, if we drop the requirement
of projectivity, we have a choice of two small resolutions at each $p\in \Lambda$. This will result in different birational
smooth models of $W$; they are not projective with the exception of $X$ and $\widehat X$ and
they can be obtained from $X$ by a sequence of flops.

In addition, we can construct other birational ``models'' of $W$ via flops; these are complex 3-folds birational to $W$
with Picard rank $2$. For example, we may start with $X$ and let $C\subset R_1$ be a $(-1)$-curve on $R_1$.
That is, $C$ is a smooth rational curve on $R_1$ with $K_{R_1}\cdot C = -1$.
Since $R_1$ is the blowup of $S_1$ at $\lambda\ge 9$ points,
it is well known that there are infinitely many $(-1)$-curves on $R_1$. By the exact sequence
\begin{equation}\label{E800}
\xymatrix{
0 \ar[r] & N_{C/R_1} \ar[r] \ar@{=}[d] & N_{C/X} \ar[r] & N_{R_1/X}\bigg|_{C} \ar[r] \ar@{=}[d] & 0\\
         & \CO_C(-1)                  &                & \CO_C(-1)
}
\end{equation}
we see that $N_{C/X} \isom \CO_C(-1)\oplus \CO_C(-1)$. Hence $X$ is locally isomorphic to an analytic neighborhood of the
zero section of $N_{C/X}$ along $C$. Consequently, we can contract $C$ to a rational double point, whose resolutions lead to
a flop $X\dashrightarrow X'$.

We can even construct a sequence of flops along a chain of rational curves. In the more general setting,
let $X$ be a flat projective family of surfaces over $\Delta$. Suppose that $X$ is smooth, $X_0$ has simple
normal crossing and there is a chain of rational curves $G_1\cup G_2\cup ...\cup
G_n\subset X_0$ satisfying that
\begin{itemize}
\item $G_k\isom \PP^1$;
\item $G_k\cap G_l = \emptyset$ for $|k - l| > 1$;
\item $R_k \ne R_{k+1}$, where $R_k$ is the unique component of $X_0$ containing $G_k$;
\item $G_k(X_0- R_k) = q_{k-1} + q_{k}$ for $1\le k\le n$,
where $q_1,q_2,...,q_n$ are $n$ distinct points and $q_0 = \emptyset$;
\item $N_{G_1/R_1} = \CO_{G_1}(-1)$ and $N_{G_k/R_k} = \CO_{G_k}$ for $k\ge 2$.
\end{itemize}
There is a flop $X_1 = X \dashrightarrow X_2$ about $G_1$. The proper transform of $G_2$ in $X_2$, which we still denote by
$G_2$, is a $(-1)$-curve on the proper transform of $R_2$.
So there is a flop $X_2\dashrightarrow X_3$ about $G_2$. Continuing
this process, we obtain a sequence of flops
$X_1\dashrightarrow X_2\dashrightarrow ... \dashrightarrow X_n\dashrightarrow X_{n+1}$.

\subsection{Resolution of indeterminacy}\label{SS002}

Suppose that there is a nontrivial dominant rational map $\phi_t: X_t\dashrightarrow X_t$
for every $t\ne 0$ given by a linear series in $|lL|$ for some fixed positive integer $l$. 
As mentioned in the previous section, we can extend it to a dominant rational map $\phi: X\dashrightarrow X$,
after a base change, with the commutative diagram
\begin{equation}\label{E000}
\xymatrix{
X\ar@{-->}[r]^\phi \ar[d] & X \ar[dl]\\
\Delta}
\end{equation}
Note that after a base change, $X$ is locally given by
\begin{equation}\label{E902}
xy = t^m
\end{equation}
for some positive integer $m$ at every point $p\in D$. 
So $X$ is $\BQ$-factorial and has canonical singularities along $D$.
In particular, the divisors $R_i\subset X$ are $\BQ$-Cartier ($mR_i$ are Cartier) and
$\Pic_\BQ(X)$ is generated by $L$ and $R_i$, i.e.,
\begin{equation}\label{E605}
\Pic_\BQ(X) = \BQ L \oplus \BQ R_1.
\end{equation}

The indeterminacy of $\phi$ can be resolved by a sequence of blowups along smooth centers.
Let $f: Y\to X$ be the resulting birational regular map with the commutative diagram \eqref{E906}.
We write $E_f = \sum E_k\subset Y$ as the union of all the exceptional divisors of $f$ and (for convenience)
the proper transforms $\wt{R}_i\subset Y_0$ of $R_i$ under $f$ for $i=1,2$.

Obviously, $f: Y_t\to X_t$ is a sequence of blowups at points for $t\ne 0$.
We call an irreducible exceptional divisor $E\subset E_f$ 
a ``horizontal'' exceptional divisor of $f$ if $E\not\subset Y_0$. 
The fiber $E_t$ of a horizontal exceptional divisor
$E$ over every $t\ne 0\in\Delta$ is a disjoint union of $\PP^1$'s, which are the
exceptional curves of $f: Y_t\to X_t$. So $f(E)$ is a multi-section of $X/\Delta$.
After a suitable base change, $f(E)$ becomes a union of sections of $X/\Delta$. Consequently,
$E/\Delta$ becomes a family of $\PP^1$'s over $t\ne 0$. In addition, we can make $Y_0$ into a divisor of simple normal
crossing after a base change by the stable reduction theorem in \cite{KKMS}. In summary,
with a suitable base change, i.e., an appropriate choice of $m$ in \eqref{E902}, we may assume that 
\begin{itemize}
\item $Y$ is smooth and $Y_0$ has simple normal crossing;
\item $E_t\isom \PP^1$ at $t\ne 0$ for every horizontal exceptional divisor $E$ of $f$;
\item $E_f$ has simple normal crossing;
\item the map $Y\to W$ (via $Y\xrightarrow{f} X \to W$) also factors through $\widehat X$ 
(see \ref{SS004}) and the corresponding map
${\widehat f}: Y\to {\widehat X}$ also resolves the rational map $\widehat \phi: {\widehat X}\dashrightarrow {\widehat X}$;
so we have the following commutative diagram:
\begin{equation}\label{E080}
\xymatrix{ & X \ar[d]\ar@{-->}[r]^{\phi} & X\ar[d]\\
Y \ar[r] \ar[ur]^{f} \ar[dr]_{\widehat f} \ar@/^3pc/[urr]^{\varphi} \ar@/_3pc/[drr]_{\widehat \varphi}
& W \ar@{-->}[r] & W\\
& {\widehat X} \ar[u] \ar@{-->}[r]^{\widehat\phi} & {\widehat X}\ar[u] 
}
\end{equation}
\end{itemize}

The reason that we need \eqref{E080} will be clear later.

Let $\omega_{X/\Delta}$ and $\omega_{Y/\Delta}$ be the relative dualizing sheaves of $X$ and $Y$ over $\Delta$,
respectively. The following identity plays a central role in our argument:
\begin{equation}\label{E909}
\omega_{Y/\Delta} = f^* \omega_{X/\Delta} + \sum \mu_k E_k = \sum \mu_k E_k = \varphi^*
\omega_{X/\Delta} + \sum \mu_k E_k
\end{equation}
for some integers $\mu_k = \mu(E_k) = a(E_k, X)$, which are the discrepancies of $E_k$ with respect to $X$. 
Note that $\omega_{X/\Delta}$ is trivial.
For convenience, we define $\mu(\wt{R}_i) = 0$. 

Since $X$ has at worse canonical singularities, we see that $\mu(E)\ge 0$ for all $E\subset E_f$.
And we claim that

\begin{prop}\label{PROP100}
For a component $E\subset Y_0$, 
\begin{equation}\label{E802}
\mu(E) > 0 \Rightarrow  \varphi_* E = 0.
\end{equation}
\end{prop}

To see this, we apply the following simple observation. 

\begin{lem}\label{LEM101}
Let $X/\Delta$ and $Y/\Delta$ be two flat families of complex analytic varieties of the same dimension 
over the disk $\Delta$. Suppose that $X$ has reduced central fiber $X_0$ and $Y$ is smooth.
Let $\varphi: Y\to X$ be a proper
surjective holomorphic map preserving the base. Let $S\subset Y_0$ be a reduced irreducible component of $Y_0$ with
$\varphi_* S \ne 0$. Suppose that $\varphi$ is ramified along $S$ with ramification index $\nu > 1$.
Then $S$ has multiplicity $\nu$ in $Y_0$. In particular, $Y_0$ is nonreduced along $S$.
\end{lem}

\begin{proof}
The problem is entirely local.
Let $R = \varphi(S)$, $q$ be a general point on $S$ and $p = \varphi(q)$. Let $U$ be an analytic open neighborhood of
$p$ in $X$ and let $V$ be the connected component of $\varphi^{-1}(U)$ that contains the point $q$. We may replace $X$
and $Y$ by $U$ and $V$, respectively. Then we reduce it to the case that $R$ and $S$ are the only components of $X_0$ and
$Y_0$, respectively, $R$ and $S$ are smooth and $\varphi: S\to R$ is an isomorphism, in which case the lemma follows easily.
\end{proof}

\begin{proof}[Proof of \propref{PROP100}]
If $\varphi_* E \ne 0$, then $\varphi$ is ramified along $E$ with ramification index
$\mu(E) + 1$ by \eqref{E909} and Riemann-Hurwitz. This is impossible unless $\mu(E) = 0$ by the above lemma
and \eqref{E802} follows.
\end{proof}

We let $\CS\subset Y_0$ be the union of components $E$ with $\mu(E) = 0$, i.e.,
\begin{equation}\label{E803}
\CS = \sum_{\mu_k = 0} E_k.
\end{equation}
Then it follows from \eqref{E802} that
\begin{equation}\label{E801}
\varphi_* \CS = (\deg\phi) (R_1 + R_2).
\end{equation}

Since $X$ is smooth outside of $D$, $\mu(E) > 0$ if $f(E) \not\subset D$ and $f_* E = 0$.
Consequently, we have $f(E) \subset D$ for every component $E\subset \CS$ with $f_* E = 0$. Note that $\wt{R}_i\subset
\CS$.


Actually, we can arrive at a more precise picture of $\CS$ as follows.

\subsection{Structure of $\CS$}\label{SS003}

We may resolve the singularities of $X$
by repeatedly blowing up $X$ along $R_1$. By that we mean we first blow up $X$ along $R_1$, then we blow
up the proper transform of $R_1$ and so on. Let $\eta: X'\to X$ be the resulting resolution. We see that 
\begin{equation}\label{E809}
X_0' = P_0\cup P_1\cup ...\cup P_{m-1}\cup P_m
\end{equation}
where $P_0$ and $P_m$ are the proper transforms of $R_1$ and $R_2$, respectively, $P_i$ are ruled surfaces
over $D$ for $0 < i < m$ and $P_i\cap P_j\ne \emptyset$ if and only if $|i-j| \le 1$. Note that the relative
dualizing sheaf of $X'/\Delta$ satisfies
\begin{equation}\label{E309}
\omega_{X'/\Delta} = \eta^* \omega_{X/\Delta}
\end{equation}
and hence remains trivial.

We claim that $f: Y\to X$ factors through $X'$. Again, the problem is local.  
It is enough to prove the following.

\begin{lem}\label{LEM102}
Let $X$ be the $n$-fold singularity given by $x_1 x_2 = t^m$ ($m\ge 1$)
in $\Delta_{x_1 x_2 ... x_n t}^{n+1}$ and let $\eta: X'\to X$ be the desingularization 
of $X$ obtained by repeatedly blowing up $X$ along $R_1 = \{ x_1 = t = 0\}$.
Let $f: Y\to X$ be another desingularization of $X$ with $Y_0 = f^*(X_0)$ supported
on a divisor of normal crossing. 
Here a desingularization $f:Y\to X$ of $X$ is a proper birational map from a smooth variety
$Y$ to $X$.
Then $f: Y\to X$ factors through $X'$.
\end{lem}

\begin{proof}
Let $R_2 = \{ x_2 = t = 0\}$, $D = R_1\cap R_2$ and $p$ be the origin.

Basically, we want to show that the rational map 
\begin{equation}\label{E600}
f' = \eta^{-1}\circ f: Y\dashrightarrow X'
\end{equation}
is regular. Let
\begin{equation}\label{E303}
X' = X_m \xrightarrow{} X_{m-1} \xrightarrow{} ... \xrightarrow{} X_2 \xrightarrow{} X_1 = X
\end{equation}
be the sequence of blowups over $R_1$,
where $X_k \to X_{k-1}$ is the blowup along the proper transform of $R_1$. It is easy to see that
$X_{k}$ has singularities of type $x_1x_2 = t^{m - k + 1}$. Hence we may proceed by induction on $m$ and it suffices to prove
that the rational map $Y\dashrightarrow X_2$ is regular.
So we may replace $X'$ by $X_2$. It is easy to see that $X_0' = \wt{R}_1 \cup P\cup \wt{R}_2$, where $\wt{R}_i$
are the proper transforms of $R_i$ and $P \isom D\times \PP^1$.

We can resolve the indeterminacy of $f'$ by a sequence of blowups with smooth centers.
That is, we have 
\begin{equation}\label{E304}
Z = Y_{l+1} \xrightarrow{\nu_{l}} Y_{l} \xrightarrow{\nu_{l-1}} ... \xrightarrow{\nu_2} Y_2 \xrightarrow{\nu_1}
Y_1 = Y
\end{equation}
where $\nu_{k}: Y_{k+1} \to Y_{k}$ is the blowup of $Y_{k}$ centered at a smooth
irreducible subvariety $F_{k}\subset Y_{k}$.
The resulting map $\varepsilon: Z\to X'$ is regular. Namely, we have the commutative diagram: 
\begin{equation}\label{E850}
\xymatrix{
& Z \ar[dl]_\varepsilon \ar[dr]^\nu \\
X' \ar[dr]_\eta & & Y \ar[dl]^f \ar@{-->}[ll]_{f'}\\
& X
}
\end{equation}
where $\nu = \nu_1\circ \nu_2\circ ... \circ \nu_{l}$.

In addition, we may choose the sequence of blowups that
all $Y_k$ have simple normal crossing supports on the central fiber. 

Let $f_k = f\circ \nu_1\circ \nu_2 \circ ... \circ \nu_{k-1}$
and
$f_k' = f'\circ \nu_1\circ \nu_2 \circ ... \circ \nu_{k-1}$ for $k=1,2,...,l$.

Let $E_k = \nu_k^{-1}(F_k) \subset Y_{k+1}$ be the exceptional divisor of $\nu_k$
for $k =1,2,..., l$. We will show inductively that the map $Y_{k+1}\to X'$ contracts the fibers
of $E_k/F_k$.

If $F_l$ has codimension $> 2$ in $Y_l$, then $E_l$ is a $\PP^{e}$ bundle over
$F_l$ with $e \ge 2$. 
Consider the image of $E_{l}$ under the regular map $\varepsilon: Z\to X'$.
A proper map $\PP^e\to X_0'$ must be constant if $e\ge 2$. 
Therefore, $\varepsilon$ contracts $E_l$ along the fibers of $E_l/F_l$.

Suppose that $F_l$ has codimension $2$ in $Y_l$. Then $E_l$ is a $\PP^1$ bundle over $F_l$.
Suppose that $\varepsilon$ does not contract the fibers of $E_l/F_l$. Obviously, 
$\varepsilon(E_l) \subset P$ and $\varepsilon$ maps every fiber of $E_l/F_l$ onto a
fiber of $P/D$. Hence $f_l'$ is not regular along $F_l$.

Clearly, $f_l'$ is regular at every point $q\not\in f_l^{-1}(D)$. 
So $f_l(F_l)\subset D$. Let $q\in F_l$ be a general point of $F_l$. WLOG,
we may simply assume that $p = f_l(q)$. Then the map $f_l: Y_l\to X$
induces a map on the local rings of (analytic) functions
\begin{equation}\label{E851}
f_l^\#: \CO_p \isom \BC[[x_1,x_2,... x_n ,t]]/(x_1 x_2 - t^m) \to \CO_q.
\end{equation}
Since $f_l$ is birational, $f_l^\#$ induces an isomorphism on the function fields, i.e.,
\begin{equation}\label{E852}
f_l^\#: K(\CO_p) \isom \BC((x_1,x_3,..., x_n, t)) \xrightarrow{\sim} K(\CO_q).
\end{equation}
Since $f_l'$ is not regular at $q$, we necessarily have 
\begin{equation}\label{E853}
\frac{f_l^\#(x_1)}{f_l^\#(t)} \not\in \CO_q \text{ and }
\frac{f_l^\#(t)}{f_l^\#(x_1)} \not\in \CO_q.
\end{equation}

Since $F_l$ has codimension two in $Y_l$, it is either a component of the intersection of two
distinct components of $(Y_l)_0 = Y_l\cap \{ t = 0\}$ or not contained in the intersection of
two distinct components of $(Y_l)_0$.

Suppose that $F_l$ is not contained in the intersection of
two distinct components of $(Y_l)_0$. Then $Y_l$ is locally given by $u^\alpha
= t$ at $q$ for some positive integer $\alpha$. It is easy to see that
$f_l^\#(x_1) = u^a$ and $f_l^\#(t) = u^\alpha$ for some positive
integers $a < m\alpha$. Clearly, \eqref{E853} cannot hold. Contradiction.

Suppose that $F_l$ is a component of the intersection of two distinct components of $(Y_l)_0$. 
Then $Y_l$ is locally given by $u^\alpha v^\beta = t$ at $q$ for some positive integers
$\alpha$ and $\beta$ with $\CO_q\isom \BC[[u,v,w_1, w_2, ..., w_{n-2}]]$. Hence
\begin{equation}\label{E854}
f_l^\#(x_1) = u^a v^b \text{ and } f_l^\# (t) = u^\alpha v^\beta
\end{equation}
for some integers $0\le a \le m\alpha$ and $0\le b \le m\beta$.
Obviously, \eqref{E853} holds if $a < \alpha$ and $b > \beta$ or
$a > \alpha$ and $b < \beta$. WLOG, suppose that $a < \alpha$ and $b > \beta$.
We observe that $\alpha b - a \beta \ge 2$. Then it is not hard to see
\begin{equation}\label{E855}
\begin{split}
f_l^\#(K(\CO_p)) &= \BC((u^a v^b, u^\alpha v^\beta, f_l^\#(x_3), ..., f_l^\#(x_n)))\\
& \subsetneq \BC((u,v,w_1, w_2, ..., w_{n-2})) = K(\CO_q).
\end{split}
\end{equation}
Contradiction.

In conclusion, $\varepsilon$ contracts the fibers of $E_l/F_l$ and hence
$f_l': Y_{l} \dashrightarrow X'$ is regular. Repeating this argument,
we conclude that $f_1' = f': Y = Y_1\dashrightarrow X'$ is regular. 
\end{proof}

Therefore, we have the commutative diagram
\begin{equation}\label{E856}
\xymatrix{
Y\ar[r]^\varphi\ar[d]_\varepsilon \ar[dr]^f & X\\
X' \ar[r]^\eta & X\ar@{-->}[u]_\phi
}
\end{equation}

\begin{rem}\label{REM050}
We can do the same for ${\widehat f}: Y\to {\widehat X}$. Let 
${\widehat X}'$ be the resolution of singularities of $\widehat X$ 
by repeatedly blowing up along ${\widehat R}_1$, where ${\widehat X}_0 = {\widehat R}_1\cup
{\widehat R}_2$ with ${\widehat R}_i$ the proper transform of $R_i$.
Then we have the commutative diagram
\begin{equation}\label{E398}
\xymatrix{
Y \ar@{=}[r] \ar[d]_{\widehat \varepsilon}
\ar@/_2pc/[dd]_{\widehat f} & Y\ar[d]^{\varepsilon} \ar@/^2pc/[dd]^{f}\\
{\widehat X}' \ar@{-->}[r]\ar[d]_{\widehat \eta} & X'\ar[d]^\eta\\
{\widehat X} \ar@{-->}[r]^\xi & X
}
\end{equation}
which can be put together with \eqref{E080} into
\begin{equation}\label{E176}
\xymatrix{X' \ar[r]^\eta & X \ar[d]\ar@{-->}[r]^{\phi} & X\ar[d]\\
Y \ar[u]^\varepsilon \ar[d]_{\widehat \varepsilon} \ar[r] \ar[ur]^{f} \ar[dr]_{\widehat f}
\ar[urr]_(0.7){\varphi} \ar[drr]^(0.7){\widehat \varphi}
& W \ar@{-->}[r] & W\\
{\widehat X}' \ar[r]_{\widehat \eta} & {\widehat X} \ar[u] \ar@{-->}[r]_{\widehat\phi} & {\widehat X}\ar[u]
}
\end{equation}
Let $I_p \subset P_0\isom R_1$ be the exceptional curve of $R_1\to S_1$
over a point $p\in \Lambda$ and
let $I_{i,p}$ be the fiber of $\eta: P_i\to D$ over a point $p\in D$ for $0 < i < m$. 
Then ${\widehat X}'$ can be alternatively constructed as the manifold obtained
from $X$ by the sequences of flops along
\begin{equation}\label{E363}
\bigcup_{p\in \Lambda} \bigcup_{i=0}^{m-1} I_{i,p}
\end{equation}
where we let $I_{0,p} = I_p$.
\end{rem}

Let $Q_i\subset Y$ be the proper transforms of $P_i$ under $\varepsilon$ for $i=0,1,...,m$.
Note that $Q_0 = \wt{R}_1$ and $Q_m = \wt{R}_2$. Since $X'$ is smooth, every
exceptional divisor of $\varepsilon$ has discrepancy at
least $1$. Therefore, $Q_i$ are the only components of $Y_0$ with $\mu(Q_i) = 0$.
Consequently,
\begin{equation}\label{E306}
\CS = Q_0 + Q_1 + ... + Q_{m-1} + Q_m,
\end{equation}
\begin{equation}\label{E905}
f(Q_i) = D\text{ for } 0 < i < m
\end{equation}
and
\begin{equation}\label{E807}
\varphi_* \CS = \sum_{i=0}^m \varphi_* Q_i = (\deg\phi) (R_1 + R_2).
\end{equation}
Obviously, $Q_i$ is birational to $D\times \PP^1$ for each $0 < i < m$.

Let $S$ be a component of $Y_0$. Then by \eqref{E909} and adjunction, we have
\begin{equation}\label{E214}
\omega_S = (\omega_{Y/\Delta} + S)\bigg|_S 
= \sum \mu_k E_k\bigg|_S - 
\sum_{\genfrac{}{}{0pt}{}{E_k\ne S}{E_k\subset Y_0}} E_k\bigg|_S
\end{equation}
and hence
\begin{equation}\label{E307}
\sum_{E_k\not\subset Y_0} \mu_k E_k\bigg|_S = \omega_S + 
\sum_{\genfrac{}{}{0pt}{}{E_k\ne S}{E_k\subset Y_0}} (1+\mu(S) - \mu_k) E_k\bigg|_S 
\end{equation}
Suppose that $S = Q\subset \CS$. Then \eqref{E214} becomes
\begin{equation}\label{E912}
\sum_{\genfrac{}{}{0pt}{}{E_k\ne Q}{E_k\subset Y_0}} (1 - \mu_k) E_k\bigg|_Q
= -\omega_Q + \sum_{E_k\not\subset Y_0} \mu_k E_k\bigg|_Q
\end{equation}
Suppose that $Q\ne Q_0, Q_m$. Let $Q_p\isom \PP^1$ be the fiber of $f: Q\to D$ over a general point $p\in D$. 
Clearly, we have
\begin{equation}\label{E910}
Q_p \cdot \omega_Q = -2
\end{equation}
and hence
\begin{equation}\label{E913}
\sum_{\genfrac{}{}{0pt}{}{E_k\ne Q}{E_k\subset Y_0}} (1 - \mu_k) E_k\cdot Q_p \ge 2
\end{equation}
Therefore, each $Q_j$ ($0 < j < m$) meets at least two other $Q_i$ ($0\le i \le m$) along sections of $Q_j/D$; and
since $Q_j$ is the proper transform of $P_j$, it cannot meet more than two among $Q_i$. 
So we see that $Q_i$ form a ``chain'' in the same way as $P_i$ do. More precisely, we have
\begin{itemize}
\item $Q_i$ and $Q_{i+1}$ meet transversely along a curve $D_i = Q_i\cap Q_{i+1}\isom D$ for $0\le i < m$;
\item $D_i\isom D$ are sections of $Q_i/D$ for $1\le i\le m - 1$ and $Q_{i+1}/D$ for $0\le i\le m-2$.
\item $Q_i\cap Q_j = \emptyset$ for $|i-j| > 1$.
\end{itemize}

Next, we claim that

\begin{prop}\label{PROP020}
For each $0\le i \le m$, we have
\begin{equation}\label{E808}
\text{either } \varphi_* Q_i \ne 0 \text{ or } \varphi(Q_i) = D.
\end{equation}
\end{prop}

Namely, every $Q_i$ either dominates one of $R_1$ and $R_2$ or is contracted onto $D$ by $\varphi$. Since
$\wt{R}_i$ cannot be mapped onto $D$, this implies that
\begin{equation}\label{E907}
\varphi_* \wt{R}_i \ne 0
\end{equation}
for $i=1,2$. So $\varphi$ does not contract $\wt{R}_i$, as pointed out in the very beginning of the paper.

Note that if $X$ were smooth, we would already have that $\varphi_* S\ne 0$ for all $S$ with $\mu(S) = 0$
by \eqref{E909} and Riemann-Hurwitz. However, things are a little more subtle here since $X$ is
singular.

\begin{proof}[Proof of \propref{PROP020}]
A natural thing to do is to resolve the indeterminacy of the rational map $\phi' = \eta^{-1}\circ \phi\circ \eta:
X'\dashrightarrow X'$ with the diagram
\begin{equation}\label{E810}
\xymatrix{
Y' \ar[d]_{\varepsilon'} \ar[dr]^{\varphi'}\\
X' \ar@{-->}[r]^{\phi'} \ar[d]_\eta & X'\ar[d]^\eta\\
X \ar@{-->}[r]^\phi & X
}
\end{equation}
where we can make $Y_0'$ have simple normal crossing support.
Let $Q_i'\subset Y'$ be the proper transforms of $P_i$ under $\varepsilon'$.
Obviously, $Q_i'$ are the proper transforms of $Q_i$ under the rational map
$\varepsilon^{-1} \circ \varepsilon': Y'\dashrightarrow Y$.
To show that \eqref{E808} holds for $Q_i$, it suffices to show that the same thing holds for $Q_i'$ when we map $Y'$ to
$X$ via $\eta\circ \varphi'$.

Let $\mu'$ be the discrepancy function corresponding to the map $\eta\circ \varepsilon'$, i.e.,
\begin{equation}\label{E183}
\omega_{Y'/\Delta} = (\varepsilon')^* \eta^* \omega_{X/\Delta} + \sum \mu'(E) E
\end{equation}
where $E$ runs through all exceptional divisors of $\eta\circ\varepsilon'$ and all components of $Y_0'$.
By \eqref{E309}, we see that $\mu'(Q_i') = 0$ for all $0\le i\le m$.
Then we have $\varphi_*' Q_i' \ne 0$ by Riemann-Hurwitz and the fact that $X'$ is smooth.
So each $Q_i'$ dominates some $P_j$ via $\varphi'$.
If $Q_i'$ dominates $P_0$ or $P_m$, then $Q_i'$ dominates $R_1$ or $R_2$ via $\eta\circ\varphi'$;
if $Q_i'$ dominates $P_j$ for some $0 < j < m$, then $\eta(\varphi'(Q_i')) = D$.
This proves \eqref{E808}.
\end{proof}

\subsection{The map $\CS\to X_0$}

Let us consider the restriction of $\varphi$ to $\CS$, i.e., $\varphi_\CS: \CS\to X_0$.
Again by \eqref{E909} and adjunction,
\begin{equation}\label{E812}
\begin{split}
\omega_\CS &= (\omega_{Y/\Delta} + \CS)\bigg|_\CS = (\omega_{Y/\Delta} - \T)\bigg|_\CS\\
&= 
\sum_{E_k\not\subset Y_0} \mu_k E_k\bigg|_\CS + \sum_{E_k\subset \T} (\mu_k - 1) E_k \bigg|_\CS\\
&=
\varphi_\CS^* \omega_{X_0} + \sum_{E_k\not\subset Y_0} \mu_k E_k\bigg|_\CS + \sum_{E_k\subset \T} (\mu_k - 1) E_k \bigg|_\CS
\end{split}
\end{equation}
where
\begin{equation}\label{E959}
\T = Y_0 - \CS = \sum_{\genfrac{}{}{0pt}{}{\mu_k > 0}{E_k\subset Y_0}} E_k.
\end{equation}
This gives us the discriminant locus of $\varphi_\CS$. It is also easy to see the following from \eqref{E812}.

\begin{prop}\label{PROP000}
If $i$ and $j$ are two integers satisfying that $0\le i < j\le m$, $\varphi_* Q_i \ne 0$, $\varphi_* Q_j \ne 0$ and
$\varphi_* Q_\alpha = 0$ for all $i < \alpha < j$, then we have either $\varphi(Q_i) = R_1$ and $\varphi(Q_j) = R_2$
or $\varphi(Q_i) = R_2$ and $\varphi(Q_j) = R_1$; in other words, $Q_i$ and $Q_j$ cannot dominate the same $R_n$ via
$\varphi$ for $n=1,2$. As a consequence, $\varphi(D_i) = D$ for all $0\le i < m$.
\end{prop}

\begin{proof}
We leave the proof to the readers.
\end{proof}

\subsection{Invariants $\alpha_i$ and $\beta_i$}

Let $Q = Q_i$ be a component of $\CS$. Suppose that $Q$ dominates $R = R_j$ for some $1\le j\le 2$.
Let $\varphi_{Q_i} = \varphi_{Q}: Q\to R$ be the restriction of $\varphi$ to $Q$. 
We may put \eqref{E912} in the form of log version Riemann-Hurwitz 
\begin{equation}\label{E877}
\begin{split}
&\quad\omega_Q + D_{i-1} + D_i\\
&= 
\varepsilon_Q^*(\omega_{P} + \varepsilon_* D_{i-1} + \varepsilon_* D_i) +
\sum_{E_k\subset \T} (\mu_k - 1) E_k\bigg|_Q
+ \sum_{E_k\not\subset Y_0} \mu_k E_k\bigg|_Q\\
&=
\varphi_Q^*(\omega_{R} + D) + \sum_{E_k\subset \T} (\mu_k - 1) E_k\bigg|_Q
+ \sum_{E_k\not\subset Y_0} \mu_k E_k\bigg|_Q
\end{split}
\end{equation}
where we set $D_{-1} = D_m = \emptyset$ and $\varepsilon_{Q_i} = \varepsilon_Q: Q\to P_i = P$
is the restriction of $\varepsilon$ to $Q$.
We let $\alpha_i$ and $\beta_{i-1}$ be the multiplicities of $D_i$ and $D_{i-1}$ in
$\varphi_{Q_i}^* D$, respectively.
Here we set $\alpha_i = \beta_{i-1} = 0$ if $\varphi_* Q_i = 0$. Obviously, $\alpha_i$ and
$\beta_{i-1}$ are the ramification indices of $\varphi_{Q_i}$ along $D_i$ and $D_{i-1}$,
respectively. We claim that
\begin{equation}\label{E908}
\deg \varphi_{Q_i} = \alpha_i \deg\varphi_{D_i} + \beta_{i-1}\deg\varphi_{D_{i-1}}
\end{equation}
where $\varphi_{D_i}: D_i\to D$ is the restriction of $\varphi$ to $D_i$ and we set
$\alpha_m = \beta_{-1} = 0$. This is a consequence of the following observation.

\begin{prop}\label{PROP022}
Let $Q = Q_i\subset \CS$ be a component satisfying $\varphi_* Q\ne 0$. Then
$\varphi_Q^{-1}(D) = D_{i-1}\cup D_i\cup \Gamma$ with $\varphi_* \Gamma = 0$.
This still holds if we replace $\varphi$ by ${\widehat\varphi}$.
\end{prop}

\begin{proof}
Otherwise, there is an irreducible curve $\Gamma\subset Q$ such
that $\Gamma\ne D_{i-1}$, $\Gamma\ne D_i$ and $\varphi(\Gamma) = D$.
So $\Gamma$ is not rational and hence $\varepsilon_*\Gamma \ne 0$.
Let $\Sigma$ be the union of components of $Y_0$ that dominate $G = \varepsilon(\Gamma)$ via
$\varepsilon$.
Let $q$ be a general point on $\Gamma$, $p = \varepsilon(q)$ and 
$J = \varepsilon^{-1}(p)$. By Zariski's main theorem, $J$ is connected and hence $\Sigma$ is connected.
Let $J_1\subset J$ be the component of $J$ containing $q$ and 
Let $\Sigma_1\subset\Sigma$ be the component of $Y_0$
containing $J_1$. Obviously, $\Gamma\subset \Sigma_1$ and hence $D\subset \varphi(\Sigma_1)$.
And since $\Sigma_1\subset\T$, $\varphi_* \Sigma_1 = 0$. Therefore, $\varphi(\Sigma_1) = D$.
And since $J_1\isom \PP^1\subset \Sigma_1$, $\varphi_* J_1 = 0$
and $\varphi$ contracts $\Sigma_1$ onto $D$ along the fibers of $\varepsilon: \Sigma_1\to G$.
Let $J_2\ne J_1\subset J$ be a component
of $J$ with $J_1\cap J_2\ne \emptyset$ and let $\Sigma_2\subset\Sigma$ be the component of $Y_0$ containing
$J_2$. Then $\Sigma_2$ meets $\Sigma_1$ along a multi-section of $\Sigma_1/G$. Therefore,
$\varphi(\Sigma_2) = D$ and $\varphi_* J_2 = 0$ by the same argument as before. We can argue this way
inductively that $\varphi_* J = 0$ and $\varphi(\Sigma^\circ) = D$ for every component
$\Sigma^\circ\subset\Sigma$. Let $r = \varphi(q)$ and
$K$ be the connected component of $\varphi^{-1}(r)$ containing the point $q$.
Obviously, $J\subset K$. We claim that $J = K$. Otherwise, there is a component $K^\circ\subset K$ such that
$K^\circ \not\subset J$ and $K^\circ \cap J\ne \emptyset$. Let $T$ be a component of $Y_0$ containing
$K^\circ$. Obviously, $T\not\subset\Sigma$; otherwise, we necessarily have $\varepsilon(K^\circ) = p$ and
$K^\circ\subset J$. Also we cannot have $T = Q$; otherwise, $K^\circ\subset Q$, $q\in K^\circ$ and
$\varphi_* K^\circ = 0$, which is impossible for a general point $q\in\Gamma$.
We cannot have $T = Q'$ for some $Q'\ne Q\subset\CS$, either,
since $p\in \varepsilon(K^\circ)\subset \varepsilon(T)$. 
Therefore, $T\subset\T$.
Since $K^\circ\cap J\ne\emptyset$, $T\cap \Sigma\ne \emptyset$. If $T$ and $\Sigma$ meet along a
multi-section of $\Sigma/G$, $\varepsilon(T) = G$, which is impossible as we have proved that 
$T\not\subset\Sigma$. Therefore, $T\cap \Sigma$ is contained in the fibers of $\Sigma/G$. And since
$T\cap J\ne\emptyset$, $T\cap\Sigma$ contains a component of $J$, which is impossible for a general point
$p\in G$. Therefore, $J = K$.

Let $U\subset X$ be an analytic open neighborhood
of $r$ in $X$ and $V\subset Y$ be the connected component of $\varphi^{-1}(U)$ containing
$J$. Since $\varepsilon(J) = p\not\in P_{i-1}\cup P_{i+1}$,
$V\cap \CS = V\cap Q$. Therefore, $\varphi_* M = 0$ for all components
$M$ of $V_0$ with $M\ne V\cap Q$. So $V_0$ cannot dominate $U_0$. Contradiction.
\end{proof}

Actually $\alpha_i$ and $\beta_{i-1}$ are determined as follows.
 
\begin{prop}\label{PROP026}
Let $0 \le i < j\le m$ be two integers with the properties that $\varphi_* Q_i \ne 0$,
$\varphi_* Q_j \ne 0$ and $\varphi_* Q_k = 0$ for all $i < k < j$. Then
\begin{equation}\label{E609}
\alpha_i = \beta_{j-1} = \frac{m}{j-i}.
\end{equation}
\end{prop}

\begin{proof}
Let $q$ be a general point on $D_i$ and $U\subset X$ be an analytic open neighborhood of
$\varphi(q)$ in $X$. Let $V\subset Y$ be the connected component of $\varphi^{-1}(U)$
containing $q$. Let $T\subset V_0$ be a component
of $V\cap \T$. Since $q$ is a general point on $D_i$, it is easy to see that
$\varepsilon(T)\cap (P_i \cup P_j) = \emptyset$ and $\varepsilon(T)\subset P_k$ for
some $i < k < j$. Indeed, $\varepsilon(T)$ is a multi-section of $P_k/D$, $T$ is
a $\PP^1$ bundle over $\varepsilon(T)$ and $\varphi$
contracts the fibers of $T/\varepsilon(T)$ and maps $T$ onto $D\cap U$. Therefore,
$\varepsilon: V\to V'=\varepsilon(V)$ is proper and $U'$ is open in $X'$. In addition, since
$T$ is contracted by $\varphi$ along the fibers of $T/\varepsilon(T)$, the rational map
$\phi = \varphi\circ \varepsilon^{-1}: V'\dashrightarrow U$ is actually regular.
Furthermore, $\phi$ contracts $P_k$ for all $i < k < j$ and hence we have the diagram
\begin{equation}\label{E618}
\xymatrix{
V' \ar[d]_\tau \ar[r]^\phi & U\\
V''\ar[ur]
}
\end{equation}
where $\tau: V'\to V''$ is the birational map contracting all $P_k$ for $i < k < j$. That is,
$V''$ is the threefold given by $xy = t^{j-i}$ in $\Delta_{xyzt}^4$. The map
$\phi\circ \tau^{-1}$ is regular and finite and sends $V'' = \{ xy = t^{j-i}\}$ onto $U = \{ xy=t^m\}$
while preserving the base $\Delta = \{|t| < 1\}$.
It has to be the map sending $(x, y, z, t)$ to $(x^a, y^a, z, t)$ with $a = m/(j-i)$. It follows
that $\alpha_i = \beta_{j-1} = a$.
\end{proof}

\begin{cor}\label{COR003}
The following holds:
\begin{itemize}
\item $\alpha_i \ne 1$ and $\beta_{i-1} \ne 1$ for all $0 < i < m$.
\item If $\deg \varphi_{Q_0} = 1$ or $\deg\varphi_{Q_m} = 1$,
then $\deg\varphi = 1$.
\end{itemize}
\end{cor}

\begin{proof}
The first statement follows directly from \propref{PROP026}.

If $\deg \varphi_{Q_0} = 1$, then $\alpha_0 = \deg\varphi_{D_0} = 1$. By \propref{PROP026}, we must have
$\beta_{m-1} = 1$ and $\varphi_* Q_k = 0$ for all $0 < k < m$. Hence
$\deg \varphi_{D_m} = \deg\varphi_{D_0} = 1$ and $\deg \varphi_{Q_m} = 1$ by \eqref{E908}.
It follows that $\deg\varphi = 1$.

Similarly, we can show that $\deg\varphi = 1$ if $\deg \varphi_{Q_m} = 1$.
\end{proof}

\begin{cor}\label{COR006}
The following are equivalent:
\begin{itemize}
\item $\alpha_0 = 1$.
\item $\beta_{m-1} = 1$.
\item $\alpha_i = 1$ for some $0\le i \le m-1$.
\item $\beta_j = 1$ for some $0\le j\le m-1$.
\item $\varphi_* Q_i = 0$ for all $1\le i \le m-1$.
\item $\deg \phi_{R_1} = \deg\phi$, where $\phi_{R_1}$ is the restriction of $\phi$ to $R_1$.
\item $\deg \phi_{R_2} = \deg\phi$, where $\phi_{R_2}$ is the restriction of $\phi$ to $R_2$.
\end{itemize}
\end{cor}

\begin{proof}
This is more or less trivial.
\end{proof}

\section{Some local results}\label{SEC004}

Before proceeding any further,
we will first prove a few lemmas of local nature. Impatient readers can skip this section and only refer back
when they are needed.

The first is basically Lemma 2.2 in \cite{Ch}.

\begin{lem}\label{LEM008}
Let $X$ be the $n$-fold defined by $x_1 x_2 = t^m$ in $\Delta_{x_1x_2...x_nt}^{n+1}$ for some
integer $m > 0$,
$\C$ be a flat family of curves over the disk $\Delta = \{|t| < 1\}$ and
$\varepsilon: \C \to X$ be a proper map preserving the base $\Delta$.
If $\varepsilon_* \C \ne 0$, then there is a component $\Gamma_i\subset \C_0$ for each $i=1,2$ such
that $\varepsilon_* \Gamma_i \ne 0$ and $\varepsilon(\Gamma_i) \subset R_i = \{ x_i = t = 0\}$.
\end{lem}

\begin{proof}
Since $\varepsilon_* \C\ne 0$, $\dim \varepsilon(\C) = 2$. Note that $X$ is $\BQ$-factorial and
$R_i$ is a $\BQ$-Cartier divisor on $X$. So $\dim (R_i\cap \varepsilon(\C)) = 1$. It follows that
there is a component $\Gamma_i\subset\C_0$ such that $\varepsilon_* \Gamma_i\ne 0$ and
$\varepsilon(\Gamma)\subset R_i$.
\end{proof}

In the above lemma, when $m = 1$ and $\C_0$ is nodal, we can say much more.

\begin{lem}\label{LEM001}
Let $X$ be the $n$-fold defined by $x_1 x_2 = t$ in $\Delta_{x_1x_2...x_nt}^{n+1}$,
$\C$ be a flat family of curves over the disk $\Delta = \{|t| < 1\}$ and
$\varepsilon: \C \to X$ be a proper map preserving the base $\Delta$. Suppose that $\C$ is smooth,
$\C_0$ has normal crossing and $\varepsilon_* \C\ne 0$.
Then
\begin{enumerate}
\item $\varepsilon(\Gamma)\not\subset D = \{x_1 = x_2 = 0\}$ for every component $\Gamma\subset\C_0$;
\item $\varepsilon_* \Gamma \ne 0$ for every component $\Gamma\subset \C_0$ and $\varepsilon(\Gamma)\cap D\ne\emptyset$;
\item $\Gamma$ meets $\varepsilon^* R_2$ transversely for every component $\Gamma\subset\C_0$ satisfying that
$\varepsilon(\Gamma)\subset R_1$, where $R_1 = \{ x_1 = 0\}$ and $R_2 = \{ x_2 = 0\}$;
\item if there is a component $\Gamma_1\subset\C_0$ and a point $q\in \Gamma_1$
satisfying that $\varepsilon(\Gamma_1)\subset R_1$ and $p = \varepsilon(q)\in D$,
then there is a component $\Gamma_2\subset \C_0$ satisfying that
$\varepsilon(\Gamma_2)\subset R_2$ and $q\in\Gamma_1\cap \Gamma_2$.
\end{enumerate}
\end{lem}

\begin{proof}
Suppose that there is a component $\Gamma\subset \C_0$ such that $\varepsilon(\Gamma)\subset D$.
Since $\C_0$ is reduced along $\Gamma$,
there is a section $B$ of $\C/\Delta$ such that $B\cap \Gamma\ne \emptyset$.
Since $\varepsilon$ preserves the base, $\varepsilon(B)$ is a section of $X/\Delta$ that meets
$D$. This is impossible as $X$ is smooth. This proves both (1) and (2).

Let $p = \{ x_1 = x_2 = ... = x_n = t = 0\}$ be the origin.
WLOG, we may assume that $\varepsilon: \C\to X$ has connected fibers, 
$p\in \varepsilon(\Gamma)$ for every component $\Gamma\subset \C_0$
and $p$ is the only point in the intersection $\varepsilon(\C)\cap D$.

By (2), $\varepsilon^{-1}(p)$ is a finite set of points.
So $\varepsilon^{-1}(p) = \{q\}$ consists of a single point $q\in \C_0$.
Consequently, $\C_0$ has at most two components. Obviously, $\varepsilon^* R_i \ne 0$
since $\varepsilon^{-1}(R_i)\cap \C_0 \ne \emptyset$. Therefore, $\C_0 = \Gamma_1\cup \Gamma_2$ consists
of two components $\Gamma_i$ with
$\varepsilon(\Gamma_i)\subset R_i$. And since $\varepsilon$ preserves the base, $\varepsilon^*
R_i$ is reduced and hence $\varepsilon^* R_i = \Gamma_i$. Then it is clear that
\begin{equation}\label{E650}
\Gamma_2\cdot \varepsilon^* R_1 = \Gamma_1 \cdot \varepsilon^* R_2 = \Gamma_1\cdot \Gamma_2 = 1
\end{equation}
and (3) and (4) follow.
\end{proof}

\begin{lem}\label{LEM005}
Let $X$ and $Y$ be two flat families of analytic varieties of dimension $n-1$ over $\Delta = \{|t| < 1\}$, 
where $X\isom \{ x_1 x_2 = t^m\}\subset \Delta_{x_1 x_2 ... x_n t}^{n+1}$ for some integer $m\ge 1$ and
$Y$ is smooth with simple normal crossing central fiber.
Let $\C\subset Y$ be a flat family of curves over $\Delta$ cut out by the general members of $n-2$ base point free
linear systems on $Y$ and let $\varepsilon: Y\to X$ be a proper birational map preserving the base.
Then $\Gamma$ meets $\varepsilon^* R_2$ transversely for every component $\Gamma\subset\C_0$ satisfying that
$\varepsilon(\Gamma)\subset R_1$ and $\varepsilon(\Gamma)\not\subset D$, 
where $R_1 = \{ x_1 = 0\}$, $R_2 = \{ x_2 = 0\}$ and $D = \{ x_1 = x_2 = t = 0\}$.
\end{lem}

\begin{proof}
By \lemref{LEM102}, $\varepsilon$ factors through $X'$, where $X'$ is the desingularization of $X$ by repeatedly blowing up
along $\{ x_1 = t = 0\}$. Applying \lemref{LEM001} to $\varepsilon: \C \to X'$, we are done.
\end{proof}

\section{The pullback $\varphi^* \LL$}\label{SEC005}

Let $\C = \varphi^* \LL\subset Y$ be the pullback of a general member
\begin{equation}\label{E206}
\LL\in |\LL_{\sigma_0,\sigma_1}| = |\sigma_0 L - \sigma_1 R_1|
\end{equation}
on $X$, where $\sigma_i$ are positive integers such that
$\LL_{\sigma_0,\sigma_1} = \sigma_0 L - \sigma_1 R_1$ is
Cartier and very ample. Note that
$\LL_{\sigma_0,\sigma_1}$ is Cartier if and only if $m | \sigma_1$
and it can be made sufficiently ample if we choose $\sigma_0 >> \sigma_1 > 0$.

Since $\varphi^* \LL$ is big and base point free and $Y_0$ has simple normal crossing,
$\C$ is smooth and
the central fiber $\C_0$ of $\C/\Delta$ is a connected curve of simple normal crossing.
First, we prove the following.

\begin{prop}\label{PROP600}
For
all $t\in \Delta$,
\begin{equation}\label{E102}
p_a(\C_t) = p_a (\C\cap \CS)
\end{equation}
where $p_a(C)$ is the arithmetic genus of a curve $C$.
\end{prop}

\begin{proof}
Clearly, since $\C\cap \CS \subset \C_0$ and $\C_0$ is connected and reduced,
\begin{equation}\label{E960}
p_a(\C_t) = p_a(\C_0) \ge p_a(\C\cap \CS).
\end{equation}
On the other hand,
\begin{equation}\label{E956}
2p_a(\C_t) - 2 = (\omega_{Y/\Delta} + \C) \C Y_t = \omega_{Y/\Delta} \C Y_t + \C^2 Y_t
\end{equation}
and
\begin{equation}\label{E957}
2p_a(\C\cap \CS) - 2 = (\omega_\CS + \C) \C\bigg|_\CS = \omega_\CS \C\bigg|_\CS + \C^2 \CS.
\end{equation}
Note that
\begin{equation}\label{E958}
\C^2 Y_t = \C^2 Y_0 = \C^2 (\CS + \T) = \C^2 \CS
\end{equation}
since $\varphi_* \T = 0$. 
So it suffices to show that
\begin{equation}\label{E961}
\omega_\CS \C\bigg|_\CS\ge
\omega_{Y/\Delta} \C Y_t = \sum_{E_k\not\subset Y_0} \mu_k E_k \C_0 = \sum_{E_k\not\subset Y_0} \mu_k E_k \C (\CS +
\T)
\end{equation}
by \eqref{E909}. Combining \eqref{E961} and \eqref{E812}, it comes down to prove that
\begin{equation}\label{E966}
\sum_{E_k\subset \T} (\mu_k - 1) E_k\CS\C \ge \sum_{E_k\not\subset Y_0} \mu_k E_k\T\C.
\end{equation}
By \eqref{E812}, the dualizing sheaf $\omega_{\C\cap\CS}$ is given by
\begin{equation}\label{E962}
\begin{split}
\omega_{\C\cap \CS} &= (\omega_\CS + \C)\bigg|_{\C\cap \CS} = \omega_\CS\bigg|_{\C\cap \CS} + (\varphi^* \LL^2)\cdot \CS\\
&= \varphi^* \omega_\LL \cdot \CS + \sum_{E_k\not\subset Y_0} \mu_k E_k\CS\C + \sum_{E_k\subset \T} (\mu_k - 1) E_k
\CS\C.
\end{split}  
\end{equation}
Clearly, \eqref{E962} gives us the ramification locus of the map $\varphi: \C\cap \CS\to \LL_0$ by Riemann-Hurwitz.

Let us consider the curve $\C\cap \T$. Since
\begin{equation}\label{E963}
\varphi_* (\C\cap \T) = \varphi_* (\varphi^*\LL \cdot \T) = \LL\cdot \varphi_* \T = 0,
\end{equation}
$\varphi$ contracts every component of $\C\cap \T$. In addition, since $\C$ is base point free, we see that
\begin{equation}\label{E302}
\C \cap E_k \ne \emptyset \text{ if and only if } \dim (\varphi(E_k)\cap Y_0) > 0.
\end{equation}
We have
\begin{equation}\label{E967}
\begin{split}
\omega_{\C\cap\T} &= (\omega_\T + \C)\C\bigg|_\T = \omega_\T \C\bigg|_\T + \C^2\T = \omega_\T
\C\bigg|_\T\\
&= (\omega_{Y/\Delta} + \T) \T\C = (\omega_{Y/\Delta} + \T) Y_0\C - (\omega_{Y/\Delta} +
\T) \CS\C\\
&= \sum_{E_k\not\subset Y_0} \mu_k E_k Y_0\C -  \sum_{E_k\not\subset Y_0} \mu_k E_k \CS\C -  \sum_{E_k\subset
\T} (\mu_k + 1) E_k \CS\C \\
&= \sum_{E_k\not\subset Y_0} \mu_k E_k \T\C - \sum_{E_k\subset \T} (\mu_k + 1) E_k \CS\C.
\end{split} 
\end{equation}
Therefore,
\begin{equation}\label{E811}
\sum_{E_k\not\subset Y_0} \mu_k E_k \T \C 
= \omega_{\C\cap\T} + \sum_{E_k\subset \T} (\mu_k + 1) E_k \CS \C.
\end{equation}

Let $p$ be a point on $\LL_0\backslash D$ and $U\subset \LL$ be an analytic open neighborhood of $p\in \LL$.
Let $V\subset \C$ be a connected
component of $\varphi^{-1}(U)$. We will show that \eqref{E966} holds when restrict to $V$. Since $\LL\cap
\varphi(E_k) \cap D = \emptyset$ for every $E_k\not\subset Y_0$, this is sufficient.
If $V\cap \T = \emptyset$, the
RHS of \eqref{E966} vanishes and there is nothing to prove. Otherwise, $V\cap \T$ is a connected component of
$\C\cap \T$. Let us assume
\begin{equation}\label{E604}
V\cap \T = \C\cap \m
\end{equation}
where $\m\subset \T$ is an effective divisor contained in $\T$. Obviously, $\varphi$ contracts $\C\cap \m$ to the
point $p$.

Restricting \eqref{E811} to $\C\cap\m$ yields
\begin{equation}\label{E968}
\sum_{E_k\not\subset Y_0} \mu_k E_k \m \C 
= \omega_{\C\cap\m} + \sum_{E_k\subset \m} (\mu_k + 1) E_k \CS \C.
\end{equation}
Let $\varphi_{V}: V\to U$ be the restriction of $\varphi$ to $V$.
By \eqref{E962}, when restricted to $V\cap \CS$,
$\varphi_{V\cap \CS}: V\cap \CS \to U_0$ is ramified along $E_k\cap \CS\cap V$ with index $\mu_k + 1$ for $E_k\not\subset
Y_0$ and with index $\mu_k$ if $E_k\subset \m$; the ramification indices at these points sum up to the degree of the
map $\varphi_{V}$ since $\varphi_V$ contracts the components of $V_0$ other than those of $V\cap \CS$, i.e., $\varphi_*(\C\cap
\m) = 0$. Therefore,
\begin{equation}\label{E813}
\deg \varphi_{V} = \sum_{E_k\not\subset Y_0} (\mu_k + 1) E_k\CS\C + \sum_{E_k\subset \m} \mu_k E_k
\CS\C
\end{equation}
in $V$. On the other hand, when restricted to a general fiber,
$\varphi_{V}$ is ramified along $E_k\cap V_t$ with index $\mu_k + 1$ for
each $E_k\not\subset Y_0$. Therefore, we have
\begin{equation}\label{E305}
\deg \varphi_V = \sum_{E_k\not\subset Y_0} (\mu_k + 1) E_k\CS\C + \sum_{E_k\subset \m} \mu_k E_k
\CS\C \ge \sum_{E_k\not\subset Y_0} (\mu_k + 1) E_k Y_t \C
\end{equation}
in $V$. Therefore,
\begin{equation}\label{E911}
\sum_{E_k\subset \m} \mu_k E_k
\CS\C \ge \sum_{E_k\not\subset Y_0} (\mu_k + 1) E_k \m \C.
\end{equation}
Combining \eqref{E968} and \eqref{E911}, we have
\begin{equation}\label{E601}
\sum_{E_k\not\subset Y_0} \mu_k E_k \m \C 
\ge \omega_{\C\cap\m} + \sum_{E_k\not\subset Y_0} (\mu_k + 1) E_k \m \C + \sum_{E_k\subset \m} E_k \CS \C
\end{equation}
and hence
\begin{equation}\label{E602}
\omega_{\C\cap\m} + \sum_{E_k\not\subset Y_0} E_k \m \C + \sum_{E_k\subset \m} E_k \CS \C \le 0.
\end{equation}
If $E_k\m \C = 0$ for all $E_k \not\subset Y_0$, there is nothing to prove. Otherwise, 
since
\begin{equation}\label{E603}
\deg \omega_{\C\cap\m} \ge -2, \sum_{E_k\not\subset Y_0} E_k \m \C \ge 1 \text{ and } \sum_{E_k\subset \m} E_k \CS \C
\ge 1,
\end{equation}
the LHS of \eqref{E602} is nonnegative and hence the equalities in \eqref{E603} must all hold.
Then \eqref{E966} clearly follows from \eqref{E968}.
\end{proof}

Indeed, we have proved more than \eqref{E102} in the above proof. In particular,
since the equalities in \eqref{E603} all hold,
we see the following:

\begin{prop}\label{PROP004}
Let $M$ be a connected component of $\C\cap \T$. Then
\begin{equation}\label{E372}
\sum_{E_k\not\subset Y_0} E_k M \le 1
\end{equation}
and
\begin{equation}\label{E373}
M \CS = 1.
\end{equation}
In other words, $M$ meets the union of horizontal exceptional divisors at no more than one point counted with multiplicity
and it meets the rest of $\C_0$ at exactly one point.
\end{prop}

\begin{rem}\label{REM011}
Both \propref{PROP600} and \ref{PROP004} hold if we replace $(\varphi, \LL, \C)$
by $(\widehat{\varphi}, \widehat{\LL}, \widehat{\C})$ (see \eqref{E080}, \eqref{E398} and \eqref{E176}),
where
$\widehat{\LL}$ is a general member of
\begin{equation}\label{E154}
|\widehat{\LL}_{\sigma_0, \sigma_1}| = |\sigma_0 \widehat{L} + \sigma_1 \widehat{R}_1|
\end{equation}
and
$\widehat{\C} = \widehat{\varphi}^* \widehat{\LL}$. Here $\widehat{L}$ is the pullback of $L$
under the map $\widehat{X} \to W$.
Note that $\widehat{\LL}_{\sigma_0, \sigma_1}$ is a very ample Cartier divisor
on $\widehat{X}$ under our assumptions that $m|\sigma_1$ and $\sigma_0 >> \sigma_1 > 0$.
\end{rem}

\section{The Push-forward $\varepsilon_*(\varphi^*\LL)$}\label{SEC003}

\subsection{Characterization of $\varepsilon_* (\C\cap \T)$}\label{SS009}

Now let us consider the push-forward $\varepsilon_* \C$. Obviously,
\begin{equation}\label{E633}
\varepsilon_* \C_0 = \varepsilon_*(\C\cap\CS) + \varepsilon_*(\C\cap\T)
\end{equation}
where every component of $\varepsilon(\C\cap\T)$ is rational by \eqref{E102}. Indeed,
$\C\cap \T$ is a disjoint union of trees of smooth rational curves and each connected component of
$\C\cap \T$ meeting the rest of $\C_0$ at a single point
by \propref{PROP600} and \ref{PROP004}.

It is easy to see that the support $\supp \varepsilon_*(\C\cap\T)$ of
$\varepsilon_*(\C\cap\T)$ is independent of the choices of $\LL$ and $\LL_{\sigma_0,\sigma_1}$:
it is the support of the union of
$\varepsilon_* (\varphi_T^{-1}(p))$ for all components $T\subset\T$ with
$\dim(\varphi(T)) = 1$ and a general point $p\in \varphi(T)$, where $\varphi_T: T\to X$ is
the restriction of $\varphi$ to $T$. 

Also we observe that since $\C$ is base point free and $\varepsilon$ maps
$\CS$ birationally onto $X_0'$, $\varepsilon_* (\C\cap \CS)$ is a linear system
with base locus of dimension $\le 0$,
i.e., consisting of isolated points, as $\LL$ varies in $|\LL_{\sigma_0,\sigma_1}|$.
In other words, $\supp\varepsilon_* (\C\cap \T)$, if nonempty, is the base locus
$\Bs(\varepsilon_*\C_0)$ of $\varepsilon_*\C_0$ in dimension one.
The base locus $\Bs(\varepsilon_*\C)$ is independent of our choice of 
$Y$, the resolution of indeterminacy of the rational map
$\varphi\circ \varepsilon^{-1}: X'\dashrightarrow X$. So
$\supp\varepsilon_*(\C\cap\T)$ is independent of
not only the choices of $\LL$ and $\LL_{\sigma_0,\sigma_1}$ but also the choice of $Y$. Indeed, it is
an invariant associated to the rational map $\varphi\circ \varepsilon^{-1}$.

As mentioned at the very beginning, $\varphi\circ \varepsilon^{-1}: X'\dashrightarrow X$ can
be resolved by resolving the base locus of $\varepsilon_* \C$. So understanding
$\varphi\circ \varepsilon^{-1}$ is more or less equivalent to
understanding $\Bs(\varepsilon_*\C)$. This shows the significance of $\varepsilon_* (\C\cap \T)$.
However, what makes $\varepsilon_* (\C\cap \T)$ really important to us is the following observation.

\begin{defn}\label{DEF000}
For a rational map $g: A\dashrightarrow B$, we call $\Exec(g)\subset A$
the exceptional locus of $g$, which is the union of all curves in the set
\begin{equation}\label{E406}
\begin{split}
\{ C: &\quad C\subset A \text{ a reduced and irreducible curve,}\\
 &\quad g \text{ is regular at the generic point of } C \text{ and } g_* C = 0 \}
\end{split}
\end{equation}
where the push-forward $g_* C$ is the closure of the push-forward $g_*(C\cap U)$ with $U\subset A$
the open set over which $g$ is regular.
\end{defn}

\begin{prop}\label{PROP023}
Let $0\le i\le m$ be an integer such $\varphi_* Q_i \ne 0$. Then
\begin{equation}\label{E059}
\Exec(\varphi\circ \varepsilon_{Q_i}^{-1}) \subset \varepsilon_* (\C\cap \T)
\end{equation}
and
\begin{equation}\label{E192}
\varphi\left(\Exec(\varphi\circ \varepsilon_{Q_i}^{-1})\right) \subset D,
\end{equation}
where $\varepsilon_{Q_i}: Q_i\to P_i$ is the restriction of $\varepsilon$ to $Q_i$.
The same holds if we replace $(\varphi, \C)$ by $({\widehat\varphi}, {\widehat\C})$.
\end{prop}

\begin{proof}
Let $G\subset P_i$ be a reduced and irreducible curve
in $\Exec(\varphi\circ \varepsilon_{Q_i}^{-1})$
and $\Gamma = \varepsilon_{Q_i, *}^{-1}(G)$ be the proper transform of $G$
under $\varepsilon_{Q_i}$. Then $\varphi_* \Gamma = 0$.
It suffices to show that $G\subset \varepsilon_* (\C\cap \T)$ and $\varphi(\Gamma)\subset D$.

Suppose that $\varphi(Q_i) = R_j$ for some $1\le j\le 2$.
For convenience, we write $Q = Q_i$, $P = P_i$, $R = R_j$, $\varphi_Q = \varphi_{Q_i}$
and $\varepsilon_Q = \varepsilon_{Q_i}$.
Since $\varphi_*\Gamma = 0$, $\Gamma$ is contained in the discriminant locus of
the map $\varphi_Q : Q \to R$.
That is, $\Gamma \subset \omega_{Q} - \varphi_{Q}^* \omega_{R}$.
By \eqref{E877},
\begin{equation}\label{E203}
\begin{split}
& \quad\Gamma \subset \omega_{Q} - \varphi_{Q}^* \omega_{R}\\
&= (\varphi_{Q}^* D - D_{i-1} - D_i) + (\omega_Q + D_{i-1} + D_i).
\end{split}
\end{equation}
Obviously, $\varepsilon_* (\omega_Q + D_{i-1} + D_i) = 0$. Therefore,
$\Gamma\subset\varphi_{Q}^* D - D_{i-1} - D_i$ and hence $\Gamma\subset \varphi_Q^{-1}(D)$.
It follows that $\varphi(\Gamma)\subset D$. That is, $\varphi$ contracts $\Gamma$ to a point on $D$.

It remains to show that $G\subset \varepsilon_* (\C\cap\T)$. Let $p$ be a general point $G$.
It suffices to show that $\varphi_*(\varepsilon^{-1} (p)) \ne 0$.

We fix a sufficiently ample divisor $B$ on $X'$. Let $B\in |B|$ be a general member passing through
$p$. The pullback $A = \varepsilon^* B\subset Y$ is a flat family of curves over $\Delta$ passing
through a point $q\in \varepsilon^{-1}(p)\cap \Gamma$. Let $\Sigma$ be the connected component
of $A\cap \T$ such that $q\in \Sigma$ and $\varepsilon(\Sigma) = p$. Obviously, $\Sigma$ is supported
on $\varepsilon^{-1}(p)$ and $\Sigma$ meets $A\cap \CS$ at the single point $q$. Let us consider
the map $\varphi: A \to X$ locally at $q$. It maps $q$ to
the point $\varphi(q) = \varphi(\Gamma)$ lying on $D$ and
the component $A\cap Q_i$ to an irreducible curve on $R_j$ passing through $\varphi(q)$. Obviously,
$\varphi(A\cap Q_i)\ne D$ for $B$ general. So we may
apply \lemref{LEM008} to conclude that $\varphi(\Sigma)$ contains an irreducible curve lying on
$R_{3-j}$ passing through $\varphi(q)$. It follows that $\varphi_*\Sigma\ne 0$ and 
hence $\varphi_*(\varepsilon^{-1} (p)) \ne 0$. 
\end{proof}

\begin{rem}\label{REM021}
The converse of \eqref{E059}, i.e.,
\begin{equation}\label{E193}
\supp\varepsilon_* (\C\cap \T)\cap P_i \subset 
\Exec(\varphi\circ \varepsilon_{Q_i}^{-1})
\end{equation}
also holds but is considerably harder to prove. We are not going to do it here since we have no use for it.
\end{rem}

\subsection{Basic properties of $\varepsilon_*(\C\cap\T)$}\label{SS010}

We start with a few basic facts about a component of $\C_0$ not contracted by $\varepsilon$.

\begin{prop}\label{PROP006}
Let $\Gamma$ be an irreducible component of $\C_0$ with $\varepsilon_* \Gamma\ne 0$. 
Suppose that $G = \varepsilon(\Gamma)\subset P_i$ for some $0\le i \le m$. Then
\begin{enumerate}
\item $\varepsilon: \Gamma\to X'$ is an immersion at every point $q\in \Gamma$ 
whose image $\varepsilon(q)\in P_{i-1}\cup P_{i+1}$, i.e., it induces an injection
on the tangent spaces $T_{\Gamma, q}\hookrightarrow T_{X',\varepsilon(q)}$;
\item at every point $q\in \Gamma$ whose image $\varepsilon(q)\in P_{i-1} \cup P_{i+1}$,
there is a component $\Gamma'\subset \C_0$ with $q\in \Gamma'$, $\varepsilon_*(\Gamma')\ne 0$ and
$\varepsilon(\Gamma') \subset P_{i-1}\cup P_{i+1}$;
\item $\varepsilon$ maps $\Gamma$ birationally onto its image $G$ if $\Gamma\subset \CS$;
\item $\Gamma = \C\cap Q_i$ if $\varphi_* Q_i \ne 0$ and $\Gamma\subset Q_i$.
\item $G$ is a fiber of $P_i/D$ if $\Gamma\subset\T$ and
$G \subset P_i$ for some $0 < i < m$.
\end{enumerate}
The same holds true if we replace $(\varphi, \C)$ by $({\widehat\varphi}, {\widehat\C})$.
\end{prop}

\begin{proof}
By \lemref{LEM001}, 
$\varepsilon^* P_{i-1}$ and $\varepsilon^* P_{i+1}$ meet $\Gamma$ transversely and (1) follows;
(2) also follows directly from \lemref{LEM001}.

Since $\C\cap Q_i$ is base point free, $\varepsilon$ maps $\Gamma\subset \C\cap Q_i$ birationally onto its image.

If $\varphi_* Q_i \ne 0$, $\C$ is big and base point free on $Q_i$. Therefore, $\C\cap Q_i$ is irreducible and
$\Gamma = \C\cap Q_i$. 

The last statement follows directly from the fact that $\Gamma$ is rational if $\Gamma\subset\C\cap \T$.
\end{proof}

A connected component $M$ of $\C\cap \T$ will meet $\C\cap \CS$ at a single point.
This fact leads to the following.

\begin{prop}\label{PROP007}
Let $M$ be a connected component of $\C\cap \T$. We write
\begin{equation}\label{E626}
M = M_0 + M_1 + ... + M_m
\end{equation}
where $\varepsilon(M_i)\subset P_i$ for $0\le i \le m$. Then
for each $0\le i < m$, either
\begin{equation}\label{E376}
\varepsilon_* M_i \cdot P_{i+1} = \varepsilon_* M_{i+1} \cdot P_{i}
\end{equation}
or
\begin{equation}\label{E377}
\varepsilon_* M_i \cdot P_{i+1} = \varepsilon_* M_{i+1} \cdot P_{i} \pm p
\end{equation}
for some $p\in P_i\cap P_{i+1}$ and \eqref{E377} holds for at most one $i$. Or equivalently,
we have either
\begin{equation}\label{E155}
f_* M_0 \cdot D = f_* M_m \cdot D
\end{equation}
or
\begin{equation}\label{E374}
f_* M_0 \cdot D = f_* M_m \cdot D \pm p
\end{equation}
for some $p\in D$, where the intersections are taken on $R_1$ and $R_2$, respectively.
The same holds true if we replace $\C$ by ${\widehat\C}$.
\end{prop}

\begin{proof}
If the intersection multiplicities $(\varepsilon_* M_i\cdot P_{i+1})_p$ and
$(\varepsilon_* M_{i+1}\cdot P_i)_p$ do not agree
at some point $p\in P_i\cap P_{i+1}$, say
\begin{equation}\label{E375}
\alpha = (\varepsilon_* M_i \cdot P_{i+1})_p - (\varepsilon_* M_{i+1} \cdot P_i)_p \ne 0,
\end{equation}
then by \lemref{LEM001},
$M$ will meet the rest of $\C_0$ at $|\alpha|$ distinct points $q_i$ with $\varepsilon(q_i) = p$ for
$1\le i\le |\alpha|$. Therefore, we must have either \eqref{E376} or \eqref{E377} and \eqref{E377} cannot hold for more than one $i$.

Since $\varepsilon_* M_i$ are supported on the fibers of $P_i/D$ for $0<i<m$, we see that
\eqref{E155} or \eqref{E374} follows.
\end{proof}

Next, we have the following key fact.

\begin{prop}\label{PROP025}
Every point $p\in f_*(\C\cap\T)\cap D$ lies in the image of 
\begin{equation}\label{E161}
\rho: \Pic(R_1) \oplus \Pic(R_2) \to \Pic(D)
\end{equation}
in $\Pic(D)$, where
$\rho$ is the map given in \eqref{E628}.
The same holds true if we replace $\C$ by ${\widehat\C}$.
\end{prop}

\begin{proof}
It is enough to show that $p\in \IM(\rho)$ for every point $p\in f_*(M)\cap D$ and every connected component
$M$ of $\C\cap \T$.

Note that $M$ is a tree of smooth rational curves. We will prove it inductively by constructing
a sequence of trees of smooth rational curves with marked points.

Let $q = M\cap \CS$. We start with $M_0 = M$. Let $\Gamma$ be a component of $M_0$ such that
$\Gamma$ has valence $1$ in $M_0$ and $q\not\in \Gamma$. Suppose that $G = \varepsilon(\Gamma)\subset P_i$.

If $G\cap (P_{i-1}\cup P_{i+1}) = \emptyset$, we simply remove $\Gamma$ from $M_0$ and
let $M_1 = M_0 - \Gamma$.
Otherwise, since $\Gamma$ has valence $1$ in $M_0$,
$\Gamma$ meets $\varepsilon^*(P_{i-1}+P_{i+1})$ transversely at exactly one point $q'$,
where $q' = \Gamma\cap \Gamma'$ with $\Gamma'$ a component of $M$.

It is clear that $f(q') = f_* \Gamma$ in $\Pic(D)$ and hence $f(q') \in \IM(\rho)$. Now we remove $\Gamma$
from $M_0$ and let $M_1 = M_0 - \Gamma$ with one marked point $q'$ on $\Gamma'$.

We continue this process to get a sequence $M_0, M_1, M_2,..., M_n$ of trees of smooth rational curves
with marked points. For each $M_a$ and a component $\Gamma\subset M_a$ with 
$\varepsilon(\Gamma)\subset P_i$ for some $i$, we have
\begin{equation}\label{E469}
\Gamma\cdot \varepsilon^*(P_{i-1} + P_{i+1}) = \sum q_j + \sum r_k + q
\end{equation}
if $q\in \Gamma$ and $\varepsilon(q)\in P_{i-1}\cup P_{i+1}$ and
\begin{equation}\label{E470}
\Gamma\cdot \varepsilon^*(P_{i-1} + P_{i+1}) = \sum q_j + \sum r_k
\end{equation}
otherwise, where $\{q_j\}$ are all the marked points of $M_a$ lying on $\Gamma$ and $\{r_k\}$ are
the intersections between $\Gamma$ and $M_a - \Gamma$ satisfying $\varepsilon(r_k)\in P_{i-1}\cup P_{i+1}$.

Inductively, we have $f(q_j)\in \IM(\rho)$ for all marked points $q_j\in M_a$. 
We construct $M_{a+1}$ by removing a component $\Gamma$ of $M_a$ which has valence $1$ and does not contain
$q$. Suppose that $q' = \Gamma\cap \Gamma'$ for a component $\Gamma'$ of $M_a$. Then by \eqref{E470},
we have $f(q')\in \IM(\rho)$ if $\varepsilon(q')\in P_{i-1}\cup P_{i+1}$, where we assume that
$\varepsilon(\Gamma)\subset P_i$.

If $\varepsilon(q')\in P_{i-1}\cup P_{i+1}$, we let $M_{a+1} = M_a - \Gamma$ with one extra marked point $q'$;
otherwise, we simply let $M_{a+1} = M_a - \Gamma$.

Continue this process and we will eventually arrive at $M_n$, which consists of a single
component $\Gamma$ passing
through $q$. By \eqref{E470}, $f(q)\in \IM(\rho)$ if $\varepsilon(q)\in P_{i-1}\cup P_{i+1}$
with $\varepsilon(\Gamma)\subset P_i$.

Clearly, a point $p\in f_*(M)\cap D$ is either $f(q)$ or $f(q')$ for a marked point $q'$ on some
$M_a$. Hence $p\in \IM(\rho)$ for all $p\in f_*(M)\cap D$.
\end{proof}

\subsection{The case $\alpha_0 = 1$}

Now we are ready to handle the case $\alpha_0 = 1$. Namely, we will prove

\begin{prop}\label{PROP010}
If $\alpha_0 = 1$, then $\deg\phi = 1$.
\end{prop}

It suffices to prove $\deg \varphi_{Q_0} = \deg\varphi_{D_0} = 1$ by \eqref{E908}
and \coref{COR003}.

\begin{prop}\label{PROP062}
Let $Q_i\subset\CS$ be a component of $Y_0$ with $\varphi_* Q_i\ne 0$ and let
$p$ be a point on $P_i\cap (P_{i-1}\cup P_{i+1})$ satisfying
\begin{equation}\label{E090}
\dim {\varphi} (\varepsilon_{Q_i}^{-1}(p)) > 0
\end{equation}
where $\varepsilon_{Q_i}: Q_i\to P_i$ is the restriction of $\varepsilon$ to $Q_i$.
Then $p\in \varepsilon_*(\C\cap\T)$. More precisely,
\begin{equation}\label{E194}
p = \varepsilon(\Gamma_1\cap \Gamma_2)
\end{equation}
for two components $\Gamma_j$ of $\C_0$ satisfying $\varepsilon_*\Gamma_j\ne 0$,
$\Gamma_1\subset Q_i$ and $\Gamma_2\subset \T$ and hence
\begin{equation}\label{E196}
p\in \varepsilon(\C\cap \T\cap Q_i).
\end{equation}
The same holds true if we replace $(\varphi, \C)$ by $({\widehat\varphi}, {\widehat\C})$.
\end{prop}

\begin{proof}
We write $P = P_i$, $Q = Q_i$, $\varphi_Q = \varphi_{Q_i}$ and
$\varepsilon_Q = \varepsilon_{Q_i}$.
The hypothesis \eqref{E090} is equivalent to saying that
the rational map $\varphi\circ\varepsilon_Q^{-1}: P\dashrightarrow X$ is
not regular at $p$.

Let
\begin{equation}\label{E095}
\Bs(\varepsilon_* (\C\cap Q)) = \Bs(\varepsilon_* \varphi_Q^* \LL) 
= \bigcap_{\LL\in|\LL_{\sigma_0,\sigma_1}|} \varepsilon_* (\C\cap Q)
\end{equation}
be the base locus of $\varepsilon_*(\C\cap Q)$ as $\LL$ varies in $|\LL_{\sigma_0,\sigma_1}|$.
Note that
\begin{equation}\label{E629}
\dim \Bs(\varepsilon_* (\C\cap Q)) \le 0.
\end{equation}

It is easy to see that the map $\varphi\circ\varepsilon_Q^{-1}$ is not regular at $p$ if and only
if $p\in \Bs(\varepsilon_* (\C\cap Q))$. Therefore, $p\in \varepsilon_* \C
= \varepsilon_* \varphi^*\LL$ for all
$\LL\in |\LL_{\sigma_0,\sigma_1}|$.

Suppose that \eqref{E194} fails.
Let $\Gamma_1 = \C\cap Q$.
Since $p\in \Bs(\varepsilon_* (\C\cap Q))$, the component
$\Gamma_1$ passes through $p$. Let $q\in \varepsilon^{-1}(p)\cap \Gamma_1$.
\lemref{LEM001} tells us there is a component $\Gamma_2$ of
$\C_0$ such that $q \in \Gamma_1\cap \Gamma_2$, $\varepsilon_*\Gamma_2\ne 0$,
$\varepsilon(\Gamma_2)\subset P_{i-1}\cup P_{i+1}$ and $\varepsilon_*\Gamma_2$ meets
$P_i$ transversely at $p$. Since \eqref{E194} fails, $\Gamma_2\subset\C\cap\CS$ and hence
$\Gamma_2\subset Q_{i-1}\cup Q_{i+1}$. It follows that
$q\in D_{i-1}\cup D_{i}$. Hence $q$ is one of finitely many points on $D_{i-1}$ and $D_i$ that
maps to $p$ via $\varepsilon$.
We conclude that $\C = \varphi^*\LL$ has a base point at $q$. This is impossible since
$\varphi^*\LL$ is base point free as $\LL$ varies in $|\LL_{\sigma_0,\sigma_1}|$.
\end{proof}

\begin{cor}\label{COR012}
Let $Q_i\subset\CS$ be a component of $Y_0$ with $\varphi_* Q_i\ne 0$. The rational
map $\varphi\circ \varepsilon_{Q_i}^{-1}$ is regular and
finite at a point $p$ on $P_i\cap (P_{i-1}\cup P_{i+1})$ if
$p\not\in \varepsilon_*(\C\cap \T)$.
The same holds true if we replace $(\varphi, \C)$ by $({\widehat\varphi}, {\widehat\C})$.
\end{cor}

\begin{proof}
By the above proposition, $\varphi\circ \varepsilon_{Q_i}^{-1}$ is regular at $p$.
It is also finite at $p$ since
$p\not\in \Exec(\varphi\circ \varepsilon_{Q_i}^{-1})$ by \propref{PROP023}.
\end{proof}

\begin{prop}\label{PROP070}
Let
\begin{equation}\label{E091}
\delta_i = \varphi_{D_{i}}\circ f_{D_{i}}^{-1}: D\to D
\end{equation}
with $f_{D_i} : D_i\to D$ the restriction of $f$ to $D_i$ for $0\le i\le m-1$.
Then
\begin{itemize}
\item $\delta_i^{-1}(\Lambda) \subset \supp(f_*(\C\cap\T)
+ f_*({\widehat\C}\cap \T))$ for $0\le i\le m-1$. More precisely, if
$\varphi_* Q_i \ne 0$, then
\begin{equation}\label{E195}
\begin{split}
\varepsilon(\varphi_{D_{i-1}}^{-1}(\Lambda)) \cup 
\varepsilon(\varphi_{D_{i}}^{-1}(\Lambda)) 
&\subset
\Exec(\varphi\circ \varepsilon_{Q_i}^{-1}) \cup 
\Exec({\widehat \varphi}\circ \varepsilon_{Q_i}^{-1})\\
&\cup \varepsilon(\C\cap \T\cap Q_i) \cup
\varepsilon({\widehat\C}\cap \T\cap Q_i).
\end{split}
\end{equation}
\item $\deg \varphi_{D_i} = \deg \delta_i = 1$ for $0\le i \le m-1$.
\end{itemize}
\end{prop}

\begin{proof}
When $\varphi_* Q_i = 0$, $Q_i$ is contracted by $\varphi$ along the fibers of $Q_i/D$. So
$\delta_{i-1} = \delta_i$. Hence it suffices to prove the proposition for $Q_i$ with $\varphi_* Q_i \ne 0$.

Let $Q = Q_i$.
Suppose that $\varphi(Q) = R_1$. For a point $p\in \Lambda$, let $\Gamma$ be a connected component
of $\varphi_{Q}^{-1}(I_p)$. Obviously,
$(D_{i-1}\cup D_i)\cap \Gamma\ne \emptyset$ and hence
\begin{equation}\label{E619}
\left(\varphi_{D_{i-1}}^{-1}(p)\cup \varphi_{D_i}^{-1}(p)\right)\cap \Gamma \ne 0.
\end{equation}
On the other hand, it is clear that every point in
$\varphi_{D_{i-1}}^{-1}(p)\cup \varphi_{D_i}^{-1}(p)$ lies on one of the connected
components of $\varphi_{Q}^{-1}(I_p)$.

We observe that ${\widehat\varphi}_{Q} : Q\to {\widehat R}_1$ factors through $\varphi_Q: Q\to R_1$
with $R_1\to {\widehat R}_1$ blowing down all $I_p$ for $p\in\Lambda$.
Therefore, ${\widehat\varphi}_*\Gamma = 0$.

Suppose that $\varepsilon_* \Gamma \ne 0$. Then
\begin{equation}\label{E092}
\varepsilon(\Gamma) \subset \Exec({\widehat \varphi}\circ \varepsilon_{Q_i}^{-1}).
\end{equation}

Suppose that $\varepsilon_* \Gamma = 0$. Let $q = \varepsilon(\Gamma)$.
Since $I_p \subset \varphi(\Gamma)$ and
$\Gamma\subset \varepsilon_{Q}^{-1}(q)$, $\varphi_* \varepsilon_{Q}^{-1}(q) \ne 0$. So 
$q\in \varepsilon(\C\cap\T\cap Q_i)$ by \propref{PROP062}.

In conclusion, we have
\begin{equation}\label{E089}
\bigcup_{p\in\Lambda}
\varepsilon(\varphi_{Q}^{-1}(I_p)) \subset \Exec({\widehat \varphi}\circ \varepsilon_{Q_i}^{-1})
\cup \varepsilon(\C\cap\T\cap Q_i)
\end{equation}
when $\varphi(Q) = R_1$.
The same argument shows
\begin{equation}\label{E086}
\bigcup_{p\in\Lambda}
\varepsilon({\widehat\varphi}_{Q}^{-1}({\widehat I}_p))
\subset \Exec(\varphi \circ \varepsilon_{Q_i}^{-1})
\cup \varepsilon({\widehat \C}\cap\T\cap Q_i)
\end{equation}
when $\varphi(Q) = R_2$, where ${\widehat I}_p$ is the exceptional curve of the blowup ${\widehat R}_2\to R_2$
over $p\in\Lambda$. Therefore, \eqref{E195} follows and
\begin{equation}\label{E197}
\delta_i^{-1}(\Lambda) \subset \supp(f_*(\C\cap\T)
+ f_*({\widehat\C}\cap \T))
\end{equation}
by \propref{PROP023}. And by \propref{PROP025}, every point $q\in \delta_i^{-1}(\Lambda)$ lies in the image
$\IM(\rho)$ of the map $\rho$.

If $\deg \varphi_{D_i} = \deg \delta_i > 1$, then $\delta_i^{-1}(p)$ contains at least two distinct points
$q_1\ne q_2\in\IM(\rho)$ for a point $p\in \Lambda$. By \eqref{E051},
$q_1 - q_2$ must be torsion. On the other hand, 
$\IM(\rho)$ is obviously torsion-free (see \eqref{E628}). 
Contradiction and hence $\deg\delta_i = 1$.
\end{proof}

\propref{PROP010} then follows easily. This settles the case $\alpha_0 = 1$.

\subsection{The case $\alpha_0\ne 1$}\label{SS011}

When $\alpha_0\ne 1$, $\varphi_* Q_i \ne 0$ for some $1 \le i \le m-1$ by \coref{COR006}.
Let $G\subset P_i$ be a general fiber of $P_i/D$. We see that the rational map
$\varphi\circ \varepsilon_{Q_i}^{-1}: P_i\dashrightarrow R_j$ is regular along $G$. We make
the key observation
\begin{equation}\label{E062}
G\cap \varepsilon_*(\C\cap\T) = \emptyset \Rightarrow G\cap \Exec(\varphi\circ \varepsilon_{Q_i}^{-1})
=\emptyset.
\end{equation}
As a consequence, we see that $\varphi\circ \varepsilon_{Q_i}^{-1}$ maps $G$ onto a curve $\Gamma$ which
only meets $D$ at two points. This line of argument leads to the following.

\begin{prop}\label{PROP027}
Suppose that $\varphi(Q_i) = R_j$ for some $1\le i\le m-1$ and $1\le j\le 2$. Let
$G\subset P_i$ be a general fiber of $P_i/D$, $\Gamma = \psi_{P_i}(G)$
and ${\widehat \Gamma} = {\widehat\psi}_{P_i}(G)$, where
$\psi = \varphi\circ \varepsilon^{-1}$,
${\widehat\psi} = {\widehat\varphi}\circ \varepsilon^{-1}$ and
$\psi_{P_i} = \varphi\circ \varepsilon_{Q_i}^{-1} : P_i\dashrightarrow R_j$ and
${\widehat \psi}_{P_i} = {\widehat\varphi}\circ \varepsilon_{Q_i}^{-1} : P_i \dashrightarrow {\widehat R}_j$
are the restrictions of $\psi$ and ${\widehat \psi}$ to $P_i$, respectively. Then
\begin{itemize}
\item $\Gamma\subset R_j$ and ${\widehat \Gamma}\subset {\widehat R}_j$
are irreducible and base point free and meet $D$ and ${\widehat D}$ set-theoretically at two points,
respectively.
\item $\Gamma\cdot D = {\widehat \Gamma}\cdot {\widehat D} \ge 2$.
\item $\Gamma\cdot I_p = 0$ on $R_j$ for all $p\in\Lambda$ if $j=1$ and
${\widehat\Gamma}\cdot {\widehat I}_p = 0$ on ${\widehat R}_j$ for all $p\in\Lambda$ if $j=2$.
\item For all curves $A\subset \Exec(\phi_{R_j})$, $A\cdot \Gamma = 0$ on $R_j$, where
$\phi_{R_j}$ is the restriction of $\phi$ to $R_j$.
\item For all curves $A\subset \Exec(\xi^{-1}\circ \phi_{R_j})$, $A\cdot \Gamma = 0$ on $R_j$, where
$\xi$ is the birational map ${\widehat X}\dashrightarrow X$ in \eqref{E398}. 
\end{itemize}
\end{prop}

\begin{proof}
We write $P = P_i$, $Q = Q_i$, $\psi_P = \psi_{P_i}$ and ${\widehat\psi}_P = {\widehat \psi}_{P_i}$.

Let $q_1 = G\cap P_{i+1}$ and $q_2 = G\cap P_{i-1}$. 
By \eqref{E062}, $\psi_P(q)\not\in D$ for all $q\in G\backslash\{ q_1, q_2\}$.
Therefore, $\Gamma$
meets $D$ set-theoretically only at $\psi_P(q_1)$ and $\psi_P(q_2)$.

Clearly, $\Gamma$ does not pass through a fixed point as $G$ varies;
otherwise, there is a curve $\Sigma\subset \Exec(\psi_P)$ with $\Sigma\cdot G \ne 0$.
Therefore, $|\Gamma|$ is base point free.
Similarly, ${\widehat \Gamma}$
meets ${\widehat D}$ set-theoretically at ${\widehat\psi}_P(q_1)$ and ${\widehat\psi}_P(q_2)$.

When $j=1$, ${\widehat \psi}_P: P\dashrightarrow {\widehat R}_1$ factors through
$\psi_P: P\dashrightarrow R_1$. 
If $\Gamma\cdot I_p \ne 0$, $p\in {\widehat \Gamma}\cap {\widehat D}$, while we have
proved that ${\widehat \Gamma}$ and ${\widehat D}$ meet only
at ${\widehat\psi}_P(q_1)$ and ${\widehat\psi}_P(q_2)$, which are general points on ${\widehat D}$ for $G$
general.
Therefore, we must have $\Gamma\cdot I_p = 0$ for all $p\in\Lambda$.
Similarly, ${\widehat\Gamma}\cdot {\widehat I}_p = 0$ for all $p\in\Lambda$ when $j=2$.

Suppose that there is a curve $A\subset \Exec(\phi_{R_j})$ such that $A\cdot\Gamma > 0$.
Since $\Gamma$ does not pass through a fixed point as $G$ varies,
$\Gamma$ meets $A$ at general points of $A$. Let us consider the rational map:
\begin{equation}\label{E198}
\xymatrix{ \phi\circ \psi: & X' \ar@{-->}[r]^\psi & X \ar@{-->}[r]^\phi & X\\
& G\ar[r]\ar[u]_\subset & \Gamma\ar[u]_\subset}
\end{equation}
Applying the above argument to $\phi^2$, we see that $\phi(\psi_P(G)) = \phi_{R_j}(\Gamma)$
is an irreducible curve on $R_k$ meeting $D$ only at $2$ points, where we assume that
$\phi(R_j) = R_k$ for some $1\le k\le 2$. Since $\Gamma$ meets $A$ at general points of $A$, $\phi_{R_j}$
sends $\Gamma\cap A$ to some points in $\phi(A)$, which lie on $D$ by
\propref{PROP023}. It follows that $\phi_{R_j}(\Gamma)$ meets $D$ at points other than
$\phi_{R_j}(\psi_P(q_1))$ and $\phi_{R_j}(\psi_P(q_2))$. Contradiction. Note that
$\psi_P(q_1)$ and $\psi_P(q_2)$ are general points on $D$ for $G$ general.
Therefore, $A\cdot \Gamma = 0$ for all $A\subset \Exec(\phi_{R_j})$. We can show the same statement
for $A\subset \Exec(\xi^{-1}\circ \phi_{R_j})$ by considering
\begin{equation}\label{E199}
\xymatrix{\xi^{-1}\circ\phi\circ \psi: & X' \ar@{-->}[r]^\psi & X \ar@{-->}[r]^\phi & X
\ar@{-->}[r]^{\xi^{-1}} & {\widehat X}\\
& G\ar[r]\ar[u]_\subset & \Gamma \ar[u]_\subset}
\end{equation}
\end{proof}

\begin{cor}\label{COR010}
Let $\Gamma_j\subset R_j$ be the union of $\psi_{P_i}(G)$ for all $P_i$ satisfying $\psi(P_i) = R_j$.
\begin{itemize}
\item If $\Gamma_1^2 > 0$, then
\begin{equation}\label{E403}
\Exec(\phi_{R_1})\cup \Exec(\xi^{-1}\circ \phi_{R_1})\subset
\bigcup_{p\in \Lambda} I_p \cup C_1
\end{equation}
where $C_1 = \emptyset$ when $g = 2$ or $g$ is odd and
$C_1\subset R_1$ is the pullback of the $(-1)$ curve on $S_1\isom \BF_1$ when $g\ge 4$ is even.
\item If $\Gamma_2^2 > 0$, then
\begin{equation}\label{E360}
\Exec(\phi_{R_2})\cup \Exec(\xi^{-1}\circ \phi_{R_2})\subset C_2
\end{equation}
where $C_2 = \emptyset$ when $g = 2$ or $g$ is odd and
$C_2\subset R_2$ is the $(-1)$ curve on $R_2\isom \BF_1$ when $g\ge 4$ is even.
\end{itemize}
\end{cor}

\begin{proof}
This follows directly from \propref{PROP027}.
\end{proof}

If
\begin{equation}\label{E300}
\bigcup_{i=1}^{m-1} \varphi(Q_i) = R_1\cup R_2
\end{equation}
and $\Gamma_j^2 > 0$ for $j=1,2$,
then both \eqref{E403} and \eqref{E360} hold by \coref{COR010}, which is the ideal situation for us.
It is not hard to see that \eqref{E300} is true if and only if one of the following two conditions holds:
\begin{itemize}
\item $\alpha_0 \ne 1$ (or equivalently $\deg\phi_{R_1} < \deg\phi$) and $\phi(R_1) \ne \phi(R_2)$; we can
put this into the form
\begin{equation}\label{E301}
R_1 + R_2\subset \phi_*(R_1) + \phi_*(R_2) \ne (\deg\phi)(R_1 + R_2).
\end{equation}
\item We have
\begin{equation}\label{E450}
\deg \phi_{R_1} + \deg \phi_{R_2} < \deg \phi.
\end{equation}
\end{itemize}

Even if both \eqref{E301} and \eqref{E450} fail, we can still make \eqref{E300} happen by replacing $\phi$
by $\phi^2$.

\begin{prop}\label{PROP011}
One of the following must be true:
\begin{enumerate}
\item $\deg\phi_{R_1} = \deg\phi$.
\item \eqref{E301} holds.
\item \eqref{E450} holds.
\item We have
\begin{equation}\label{E060}
\deg (\phi^2)_{R_1} + \deg (\phi^2)_{R_2} < \deg \phi^2 = (\deg\phi)^2
\end{equation}
where $(\phi^2)_{R_i} = \phi\circ \phi_{R_i}$ are the restrictions of $\phi^2$ to $R_i$ for $i=1,2$.
\end{enumerate}
\end{prop}

\begin{proof}
Suppose that all (1)-(3) fail. Then $\phi(R_1) = \phi(R_2) = R_j$ for some $1\le j\le 2$ and
$\deg \phi_{R_1} + \deg \phi_{R_2} = \deg\phi$. Then
\begin{equation}\label{E451}
\begin{split}
\deg (\phi^2)_{R_1} + \deg (\phi^2)_{R_2} &=
\deg (\phi_{R_j} \circ \phi_{R_1}) + \deg (\phi_{R_j} \circ \phi_{R_2})\\
&= (\deg\phi_{R_j}) (\deg \phi_{R_1} + \deg \phi_{R_2})\\
&= (\deg\phi_{R_j})(\deg\phi) < (\deg\phi)^2.
\end{split}
\end{equation}
\end{proof}

So we can always assume \eqref{E300} when $\alpha_0 \ne 1$. The more serious issue is that we might have
$\Gamma_j^2 = 0$ even if $\Gamma_j\ne 0$. This can be worked around with the following trick.

\begin{prop}\label{PROP008}
Assuming \eqref{E300}, if $\phi(R_j) = R_j$, then there is an irreducible curve
$\Gamma\subset R_j$ such that
$\Gamma$ is big and nef on $R_j$, $\Gamma\cdot I_p = 0$ for all $p\in\Lambda$ if $j=1$ and
\begin{equation}\label{E380}
A\cdot \phi^k(\Gamma) = 0
\end{equation}
for all $k\in \BZ_{\ge 0}$ and curves $A\subset \Exec(\phi_{R_j})\cup \Exec(\xi^{-1}\circ \phi_{R_j})$.
Consequently, we have \eqref{E403} if $j = 1$ and \eqref{E360} if $j = 2$.
\end{prop}

\begin{proof}
If $\Gamma_j^2 > 0$, then we are done. Otherwise, $\Gamma_j^2 = 0$ and hence $\Gamma_j = nH$ for some integer
$n>0$, where $H\cdot D = 2$ and $H$ gives a ruling of $R_j$, i.e.,
$|H|$ is a pencil giving a map $R_j\to\PP^1$.
Let $H\in |H|$ be a general member of the pencil $|H|$. 
Note that $H = \psi_{P_i}(G)$ for some $P_i$ with $\psi(P_i) = R_j$ and a general fiber $G$ of $P_i/D$.

Since $A\cdot \Gamma_j = 0$ for all curves $A\subset \Exec(\phi_{R_j})\cup \Exec(\xi^{-1}\circ \phi_{R_j})$,
\begin{equation}\label{E378}
H\cap \left(\Exec(\phi_{R_j})\cup \Exec(\xi^{-1}\circ \phi_{R_j})\right) = \emptyset.
\end{equation}
Therefore, $\Gamma = \phi_{R_j,*}(H)$ meets $D$ at two points with multiplicities $\mu$ each, where
$\mu = \alpha_0$ if $j=1$ and $\mu = \beta_{m-1}$ if $j=2$.

Note that $\phi_{R_j}\circ \eta_P = \psi_P$, where $P = P_0$ if $j=1$ and
$P = P_m$ if $j=2$. So we may identify $\phi_{R_j}$ and $\psi_P$ via the isomorphism $\eta_P: P\isom R_j$.

The key here is to prove that $\Gamma$ is reduced and hence $\Gamma^2 > 0$ since $\mu > 1$.
Now we take $H$ to be a member of $|H|$ tangent to $D$ at a point $p$, i.e., $H = 2p$ in $\Pic(D)$. It is
not hard to see that $p\not\in \IM(\rho)$. Therefore, $p\not\in f_*(\C\cap\T)$. Hence
$\phi_{R_j}$ is regular and finite locally at $p$ by \coref{COR012}. That is,
$\phi_{R_j}$ is totally ramified along $D$ at $p$ with ramification index $\mu$. It is easy to see that
$\phi_{R_j}$ locally maps $H$ at $p$ to a smooth curve $\Gamma$ tangent to $D$ at $\phi_{R_j}(p)$ with multiplicity
$2\mu$. This implies that $\Gamma$ is reduced for $H$ general and hence $\Gamma^2 > 0$ and $\Gamma$ is
big.

Similarly, ${\widehat \Gamma} = \xi_*^{-1}(\Gamma) = \xi_*^{-1}(\phi_{R_j,*}(H))$ is big and nef and meets
${\widehat D}$ at two points with multiplicities $\mu$ each. So $\Gamma\cdot I_p = 0$ for all $p\in\Lambda$
if $j=1$ by the same argument in the proof of \propref{PROP027}.

Finally, using the same argument in the proof of \propref{PROP027} again, we can conclude \eqref{E380} by
iterating $\phi$ and considering the maps
\begin{equation}\label{E379}
\xymatrix{\phi^{k+2}\circ \psi: & X' \ar@{-->}[r]^\psi & X \ar@{-->}[r]^\phi & X \ar@{-->}[r]^{\phi^k}
& X \ar@{-->}[r]^\phi & X\\
& G \ar[r] \ar[u]_{\subset} & H \ar[r]\ar[u]_{\subset} & \Gamma\ar[u]_{\subset}}
\end{equation}
and
\begin{equation}\label{E382}
\xymatrix{\xi^{-1}\circ \phi^{k+2}\circ \psi: & X' \ar@{-->}[r]^\psi & X \ar@{-->}[r]^\phi
& X \ar@{-->}[r]^{\phi^k} & X \ar@{-->}[r]^{\xi^{-1}\circ\phi} & {\widehat X}\\
& G \ar[r] \ar[u]_{\subset} & H \ar[r]\ar[u]_{\subset} & \Gamma\ar[u]_{\subset}}
\end{equation}
\end{proof}

Note that we can always assume that $\phi(R_j) = R_j$ for some $1\le j\le 2$; otherwise, if
$\phi(R_1) = R_2$ and $\phi(R_2) = R_1$, we simply replace $\phi$ by $\phi^2$.

Now we are ready to prove our main theorem when $\alpha_0 \ne 1$. We always assume \eqref{E300}.

\medskip

{\noindent \underline{Case $\phi(R_2) = R_2$ and $g$ odd or $g = 2$}}.
So we have \eqref{E360} by \propref{PROP008}. Here $R_2\isom \PP^2$ or $\BF_0$ and hence
\begin{equation}\label{E381}
\Exec(\phi_{R_2}) = \emptyset.
\end{equation}
Therefore, for a general member $H$ of an ample linear system $|H|$ on $R_2$,
$\phi_{R_2}(H)$ is a curve meeting
$D$ at $H\cdot D$ points with multiplicity $\beta = \beta_{m-1}$ each.

Suppose that there is a point $q\in R_2$ with
\begin{equation}\label{E173}
\dim (\varphi(f_Q^{-1}(q))) > 0
\end{equation}
where $f_Q: Q\to R_2$ is the restriction of $f$ to $Q = Q_m$. 
That is, $\phi_{R_2}$ is not regular at $q$.
Let $M = f_Q^{-1}(q)$.

We write
\begin{equation}\label{E395}
\varphi_Q^* D = \beta D_{m-1} + A
\end{equation}
where $\varphi_* A = 0$. By \eqref{E381}, we necessarily have $f_* A = 0$. Since
every connected component of $\supp A$ must meet $D_{m-1}$, $f(\supp A) \subset D$ 
and $f(A_1) \ne f(A_2)$ for
any two distinct connected components $A_1$ and $A_2$ of $\supp A$. As a consequence, we see that
$q\in D$ and $\varphi(M)$ meets $D$ at the unique point $\delta(q)$,
where $\delta = \delta_{m-1}: D\to D$ is the map defined in \eqref{E091}.
Hence
\begin{equation}\label{E388}
\delta(b q) = \varphi_* M
\end{equation}
in $\Pic(D)$ for some integer $b > 0$. Here we use $\delta$ for both the map $\delta: D \to D$ and
the push-forward $\delta_*: \Pic(D)\to \Pic(D)$ induced by $\delta$.

In addition, since $q\in D$,
\begin{equation}\label{E383}
q\in f_*(\C\cap \T)\cap D \subset \IM(\rho)
\end{equation}
by \propref{PROP062}.

When $g = 2$, $R_2\isom \PP^2$ and we choose $H$ to be the hyperplane divisor. Then
\begin{equation}\label{E389}
\phi_{R_2,*}(H) = \beta H \Rightarrow \delta(\beta H) = \beta H
\end{equation}
in $\Pic(D)$. Since $\varphi_* M$ and $H$ are linearly dependent in $\Pic(R_2)$, we derive
\begin{equation}\label{E390}
\delta(3 b\beta q - b\beta H) = 0 \Rightarrow b \beta H = 3b \beta q
\end{equation}
by combining \eqref{E388} and \eqref{E389} and making use of \eqref{E051} and
the fact that $\deg\delta = 1$ by \propref{PROP070}.
Since $q\in\IM(\rho)$ and $\IM(\rho)$ is torsion free, $H = 3q$ by \eqref{E390}. It is not hard to see that
such $q$ cannot lie in $\IM(\rho)$. Contradiction. Therefore, $\phi_{R_2}$ is regular and finite everywhere.
So we have
\begin{equation}\label{E391}
\phi_{R_2}^* D = \beta D \Rightarrow \beta = \beta^2 \Rightarrow \beta = 1
\end{equation}
and we are done.

When $g$ is odd, $R_2\isom \BF_0$ and we let $H_1$ and $H_2$ be the two rulings of $\BF_0 = \PP^1\times\PP^1$.
Let
\begin{equation}\label{E461}
\phi_{R_2,*}(H_1) = a_{11} H_1 + a_{12} H_2 \text{ and } \phi_{R_2,*}(H_2) = a_{21} H_1 + a_{22} H_2
\end{equation}
in $\Pic(R_2)$ for some integers $a_{ij}$. Then the previous argument shows that
\begin{equation}\label{E392}
\delta(\beta H_1) = a_{11} H_1 + a_{12} H_2 \text{ and } \delta(\beta H_2) = a_{21} H_1 + a_{22} H_2
\end{equation}
in $\Pic(D)$. It follows that
\begin{equation}\label{E394}
\pm \beta(H_1 - H_2) = \delta(\beta H_1 - \beta H_2) = (a_{11} - a_{21}) H_1 - (a_{22} - a_{12}) H_2
\end{equation}
and hence
\begin{equation}\label{E460}
a_{11} - a_{21} = a_{22} - a_{12} = \pm\beta
\end{equation}
where we make use of \eqref{E051} and the fact that $\deg\delta = 1$ again.

The effectiveness of $\phi_{R_2,*}(H_i)$ implies that $a_{ij} \ge 0$.
And 
\begin{equation}\label{E464}
a_{11} + a_{12} = a_{21} + a_{22} = \beta.
\end{equation}
This can only happen if
either
\begin{equation}\label{E462}
\phi_{R_2,*}(H_1) = \beta H_1 \text{ and } \phi_{R_2,*}(H_2) = \beta H_2
\end{equation}
or
\begin{equation}\label{E463}
\phi_{R_2,*}(H_1) = \beta H_2 \text{ and } \phi_{R_2,*}(H_2) = \beta H_1.
\end{equation}
In either case, $\phi_{R_2}$ maps a general member $H_1\in |H_1|$ onto a curve $\Gamma\in |H_1|$ or $|H_2|$
with a map of degree $\beta$. That is, $\phi_{R_2,*}(H_1) = \beta \Gamma$.

Let $H_1$ be a member of the pencil $|H_1|$ tangent to $D$ at a point $p$. We still have
$\varphi_* (f_Q^* H_1) = \beta\Gamma$. On the other hand, $H_1 = 2p$ in $\Pic(D)$ and
hence $p\not\in \IM(\rho)$
and $p\not\in f_*(\C\cap \T)$. So $\phi_{R_2}$ is regular and finite at $p$. Hence $\Gamma = \phi_{R_2}(G)$
must be smooth and tangent to $D$ at $\delta(p)$ with multiplicity $2\beta$. Contradiction.

\medskip

{\noindent \underline{Case $\phi(R_2) = R_2$ and $g\ge 4$ even}}.
Here $R_2\isom \BF_1$ and 
\begin{equation}\label{E475}
\Exec(\phi_{R_2}) \subset E = C_2,
\end{equation}
where $E$ is the $(-1)$-curve on $R_2$.

Suppose that $\Exec(\phi_{R_2}) = E$. Then
\begin{equation}\label{E452}
\varphi_*(f_Q^* H) = \phi_{R_2,*}(H) = \beta H \Rightarrow \delta(\beta H) = \beta H
\end{equation}
by \propref{PROP008}, where $H$ is the pullback of the hyperplane divisor under the blowup $\BF_1\to \PP^2$
and $\beta$ and $\delta$ are defined as above.
Let
\begin{equation}\label{E387}
\varphi_*(f_Q^* E) = b_1 H + b_2 E
\end{equation}
in $\Pic(R_2)$ for some integers $b_i$. For a general member $G\in |H - E|$, $\phi_{R_2,*}(G)$
meets $D$ at three points $\delta(q_1), \delta(q_2)$ and $\delta(p)$, where $G\cdot D = q_1 + q_2$ and
$p = E\cap D$; it meets $D$ at $\delta(q_1)$ and $\delta(q_2)$ with multiplicity $\beta$ each.
Therefore,
\begin{equation}\label{E472}
\delta(\beta (H-E)) + \delta(b_3 p) = (\beta - b_1) H - b_2 E \Rightarrow
\delta((\beta - b_3) E) = b_1H + b_2E
\end{equation}
in $\Pic(D)$ for some integer $b_3 > 0$. Note that $\beta - b_3 > 0$ due to the effectiveness of
$\varphi_*(f_Q^* E)$.

Combining \eqref{E452} and \eqref{E472} and arguing as before, we obtain
\begin{equation}\label{E473}
\pm((\beta - b_3) H - 3 (\beta - b_3) E) = \left(\beta - b_3  - 3 b_1\right) H 
- 3 b_2 E
\end{equation}
in $\Pic(R_2)$. Obviously, $b_2\ge 0$ since $\varphi_*(f_Q^*(H-E))$ is nef. Therefore, \eqref{E473} gives us
$b_1 = 0$, $b_2 = \beta - b_3$, 
\begin{equation}\label{E474}
\varphi_*(f_Q^* E) = b_2 E \text{ and } \delta(b_2 E) = b_2 E.
\end{equation}

Let $q\in D$ be a point where $\phi_{R_2}$ is not regular and let $M = f_Q^{-1}(q)$.
Then \eqref{E474} shows that $\varphi(M) = E$ for $q\in E$.

Suppose that $q\not\in E$. Using the previous argument
we can show that $q\in D$ and $\varphi(M)$ meets $D$ only at $\delta(q)$ and hence \eqref{E388} holds.

Note that $\varphi_*(f_Q^* H) - \varphi_* M = \beta H - \varphi_*M$ is nef.
Therefore, we have either $\varphi_* M\cdot E = 0$ or $E\subset \varphi(M)$. 

If $\varphi_* M \cdot E = 0$, $\varphi_* M$ is a multiple of $H$ and
we can argue as in the case $R_2\isom \PP^2$ to show that $H = 3q$ using \eqref{E388} and \eqref{E452},
which is impossible for
$q\in\IM(\rho)$. Therefore, $E\subset \varphi(M)$. And since $\varphi(M)$ meets $D$ at a single point,
we must have $\varphi(M) = E$.

In conclusion,
$\varphi(f_Q^{-1}(q)) = E$ for all points $q\in R_2$ where $\phi_{R_2}$ is not
regular.
Consequently, the rational map $g\circ \phi_{R_2}: R_2\dashrightarrow \PP^2$ is regular,
where $g: R_2\to \PP^2$ is the blow-down of $E$. Let $D' = g(D)$. Then it is easy to see that
\begin{equation}\label{E453}
(g\circ\phi_{R_2})^* D' = \beta D + \beta E \Rightarrow \beta = \beta^2
\end{equation}
and we are done.

Suppose that $\Exec(\phi_{R_2}) = \emptyset$. Let
\begin{equation}\label{E454}
\varphi_*(f_Q^* H) = a_{11} H + a_{12} E
\text{ and }
\varphi_*(f_Q^* E) = a_{21} H + a_{22} E
\end{equation}
in $\Pic(R_2)$ for some integers $a_{ij}$. Then
\begin{equation}\label{E455}
\delta(\beta H) = a_{11} H + a_{12} E
\text{ and }
\delta(\beta E) = a_{21} H + a_{22} E
\end{equation}
in $\Pic(D)$. By the same argument as in the case $R_2\isom\BF_0$, we see that
\begin{equation}\label{E456}
\pm \beta(H - 3E) = (a_{11} - 3 a_{21}) H - (3 a_{22} - a_{12}) E
\end{equation}
and hence
\begin{equation}\label{E459}
3(a_{11} - 3 a_{21}) = 3 a_{22} - a_{12} = \pm 3\beta.
\end{equation}
Combining with $3a_{11} + a_{12} = 3\beta$, we obtain
\begin{equation}\label{E457}
a_{11} = a,  a_{12} = 3\beta - 3a, a_{21} = \frac{a}3 \mp \frac{\beta}{3} \text{ and }
a_{22} = \pm\beta +\beta - a.
\end{equation}
Obviously, $\varphi_*(f_Q^* H)$ is big and nef. Hence $a_{11} > |a_{12}|$, $a_{12} \le 0$ and
\begin{equation}\label{E458}
\beta \le a < \frac{3}{2}\beta.
\end{equation}
The effectiveness of $\varphi_*(f_Q^* E)$ requires that $a_{21} + a_{22} \ge 0$. Therefore,
\begin{equation}\label{E465}
a_{21} = \frac{a}3 - \frac{\beta}{3} \text{ and }
a_{22} = 2\beta - a.
\end{equation}
The nefness of $\varphi_*(f_Q^*(H-E))$ requires that $a_{11} - a_{21} \ge a_{22} - a_{12}$. Hence
we must have $a = \beta$,
\begin{equation}\label{E466}
\varphi_*(f_Q^* H) = \beta H \text{ and }
\varphi_*(f_Q^* E) = \beta E.
\end{equation}
Note that this implies
\begin{equation}\label{E467}
\phi_{R_2,*}(H-E) = \varphi_*(f_Q^* (H - E)) = \beta(H - E)
\end{equation}
which means that $\phi_{R_2}$ maps a general member $G$ of the pencil $|H-E|$ onto
$\Gamma = \phi_{R_2}(G)\in |H-E|$ with a map of degree $\beta$. Using the same argument as in the case
$R_2\isom \BF_0$, we can show that this is impossible.

\medskip

{\noindent \underline{Case $\phi(R_1) = R_1$}}. Now we simply switch from $(X,\varphi,\phi,\varepsilon)$
to $({\widehat X}, {\widehat \varphi}, {\widehat \phi}, {\widehat \varepsilon})$ (see \eqref{E398} and
\eqref{E176}). Obviously, ${\widehat \phi}({\widehat R}_1) = {\widehat R}_1$ and ${\widehat R}_1\isom
\PP^2$, $\BF_0$ or $\BF_1$. This reduces it to the previous cases.

\end{document}